\documentclass[12pt,reqno]{amsart}
\usepackage{amssymb,amscd}
\usepackage[alphabetic]{amsrefs}
\usepackage[all]{xy}
\usepackage[headings]{fullpage}
\usepackage[T1]{fontenc}
\usepackage{libertine}
\usepackage{stackengine}
\usepackage{mathtools}
\usepackage{graphicx}

\numberwithin{equation}{section}

\swapnumbers
\theoremstyle{plain}
\newtheorem{thm}[subsection]{Theorem}
\newtheorem{thmss}[subsubsection]{Theorem}
\newtheorem{prop}[subsection]{Proposition}
\newtheorem{propss}[subsubsection]{Proposition}
\newtheorem{lemma}[subsubsection]{Lemma}
\newtheorem{cor}[subsection]{Corollary}
\newtheorem{corss}[subsubsection]{Corollary}

\newtheorem*{thm*}{Theorem}
\newtheorem{prop-def}[subsubsection]{Proposition-Definition}

\theoremstyle{definition}
\newtheorem{defn}[subsubsection]{Definition}
\newtheorem{defns}[subsubsection]{Definitions}
\newtheorem*{defn*}{Definition}

\theoremstyle{remark}
\newtheorem{rem}[subsection]{Remark}

\newtheorem{remss}[subsubsection]{Remark}

\newtheorem{s-example}[subsection]{Example}

\numberwithin{equation}{subsection}

\usepackage[OT2,T1]{fontenc}
\DeclareSymbolFont{cyrletters}{OT2}{wncyr}{m}{n}
\DeclareMathSymbol{\sha}{\mathalpha}{cyrletters}{"58}

\newcommand{\CC}{\mathcal{C}}

\newcommand{\DD}{\mathcal{D}}

\newcommand{\WW}{{\mathcal{W}}}

\newcommand{\HH}{{\mathcal{H}}}

\newcommand{\F}{\mathbb{F}}
\newcommand{\Fp}{{\mathbb{F}_p}}

\newcommand{\Fpbar}{{\overline{\mathbb{F}}_p}}

\newcommand{\Z}{\mathbb{Z}}

\newcommand{\<}{\langle}
\renewcommand{\>}{\rangle}
\newcommand{\into}{\hookrightarrow}
\newcommand{\onto}{\twoheadrightarrow}
\newcommand{\isoto}{\,\tilde{\to}\,}

\newcommand{\nodiv}{\not|}

\def\nodiv{\mathrel{\mathchoice{\not|}{\not|}{\kern-.2em\not\kern.2em|}
{\kern-.2em\not\kern.2em|}}}

\newcommand{\G}{\mathbb{G}}

\DeclareMathOperator{\im}{Im}
\DeclareMathOperator{\Ker}{Ker}

\DeclareMathOperator{\Hom}{Hom}

\DeclareMathOperator{\lcm}{lcm}

\DeclareMathOperator{\Sel}{Sel}

\def\clap#1{\hbox to 0pt{\hss#1\hss}}

\makeatletter
\newcommand*\bigcdot{\mathpalette\bigcdot@{.5}}
\newcommand*\bigcdot@[2]{\mathbin{\vcenter{
               \hbox{\scalebox{#2}{$\m@th#1\bullet$}}}}}
\makeatother

\newcommand{\D}{\mathbb{D}}
\renewcommand{\and}{\quad\text{and}\quad}
\DeclareMathOperator{\pat}{Pat}

\begin{document}
\title{On $BT_1$ group schemes and Fermat Curves}

\author{Rachel Pries}
\address{Department of Mathematics \\ Colorado State University
  \\ Fort Collins, CO~~80523 USA}
\email{pries@math.colostate.edu}
\author{Douglas Ulmer}
\address{Department of Mathematics \\ University of Arizona
 \\ Tucson, AZ~~85721 USA}
\email{ulmer@math.arizona.edu}

\date{\today}



\begin{abstract}
  Let $p$ be a prime number and let $k$ be an algebraically closed
  field of characteristic $p$.  A $BT_1$ group scheme over $k$ is a
  finite commutative group scheme which arises as the kernel of $p$ on
  a $p$-divisible (Barsotti--Tate) group.  We compare three
  classifications of $BT_1$ group schemes, due in large part to Kraft,
  Ekedahl, and Oort, and defined using words, canonical filtrations,
  and permutations.  Using this comparison, we determine the
  Ekedahl--Oort types of Fermat quotient curves and we compute four
  invariants of the $p$-torsion group schemes of these curves.
  \end{abstract}




\maketitle

\section{Introduction}
Fix a prime number $p$ and let $k$ be an algebraically closed field of
characteristic $p$.  Suppose $C$ is a smooth irreducible projective
curve of genus $g$ over $k$.  Its Jacobian ${\rm Jac}(C)$ is a
principally polarized abelian variety of dimension $g$.  The
$p$-torsion group scheme $G={\rm Jac}(C)[p]$ is a polarized $BT_1$
group scheme of rank $p^{2g}$.  The isomorphism class of $G$ is
uniquely determined by the de Rham cohomology of $C$ \cite{Oda69}.

An important way to describe the isomorphism class of $G$ is via a
combinatorial invariant called the \emph{Ekedahl--Oort type}, or E--O
type.  The E--O type is a sequence $[\psi_1, \ldots, \psi_g]$ of
integers such that $\psi_{i} -\psi_{i-1} \in \{0,1\}$ for
$1 \leq i \leq g$ (letting $\psi_0=0$).  The $p$-rank and $a$-number
of $G$ can be quickly computed from the E--O type.  The E--O type
gives key information about the stratification of the moduli space of
principally polarized abelian varieties of dimension $g$.

Recently, there has been a lot of interest in studying $p$-torsion
group schemes for curves, for example \cite{Moonen04},
\cite{ElkinPries13}, \cite{PriesWeir15},
\cite{DevalapurkarHallidaypp17}, and \cite{Moonenpp20}.  Despite this,
there are very few examples of curves for which the Ekedahl--Oort type
has been computed.  In this paper, our main result is
Theorem~\ref{thm:FEO-general}, in which we determine the Ekedahl--Oort
type for the Jacobian $J_d$ of the smooth projective curve $\CC_d$
that has affine equation $y^d=x(1-x)$, for all positive integers $d$
that are relatively prime to $p$.
 
Here is our motivation for studying the curve $\CC_d$.  First, it is a
quotient of the Fermat curve $F_d$ of degree $d$.  There has been a
lot of work on $p$-divisible groups of Fermat curves.  For example,
Yui determined their Newton polygons \cite[Thm.~4.2]{Yui80}.  Several
authors studied the $p$-ranks and $a$-numbers of Fermat curves, e.g.,
\cite{KodamaWashio88}, \cite{Gonzalez97}, and
\cite{MontanucciSpeziali18}.  It turns out that most of the
interesting features of the $p$-torsion group scheme for $F_d$ are
already present for that of $\CC_d$.

Second, the de Rham cohomology of $\CC_d$, with its Frobenius and
Verschiebung operators, can be described in a very clean way.
Typically it is complicated to discern the Ekedahl--Oort type from a
computation of these operators.  In the case of $\CC_d$, the actions
of $F$ and $V$ are given by permutation data that is easy to analyze.
For this reason, we are able to describe the mod $p$ Dieudonn\'e
module of $J_d[p]$ in several different ways and give a closed form
formula for its Ekedahl--Oort type.

Third, in our companion paper \cite{PriesUlmerBT1s}, we prove that
every polarized $BT_1$-group scheme over $k$ occurs as a direct factor
of the $p$-torsion group scheme of $J_d$ for infinitely many $d$ as
long as $p>3$.  In other words, the class of curves $\CC_d$ for
$p \nmid d$ includes all possible structures of $p$-torsion group
schemes of principally polarized abelian varieties.

To prove Theorem~\ref{thm:FEO-general}, we rely on a detailed
comparison of three classifications of $BT_1$ group schemes,
essentially due to Kraft, Ekedahl, and Oort.  We develop this
comparison in Sections~\ref{s:groups-and-modules} through
\ref{s:refinements}, working in general, not restricting to Jacobians
of curves.  The Kraft classification uses words on a two-letter
alphabet $\{f,v\}$; it interacts well with direct sums and identifies
the indecomposable objects in the category.  The Ekedahl--Oort
classification uses the interplay between $F$ and $V$ to build a
``canonical filtration'' and is well suited to moduli-theoretic
questions.  The third classification is given in terms of a finite set
$S$ partitioned into two subsets and a permutation of $S$, and it is
particularly well suited to studying Fermat curves.  The material in
these sections is fundamental to the proof of the theorem and does not
appear in a self-contained way elsewhere in the literature.

In Section~\ref{s:Homs}, we study homomorphisms between $BT_1$ group
schemes to analyze two well-known invariants, the $p$-rank and
$a$-number, and two newer invariants related to the $p$-torsion group
scheme of a supersingular elliptic curve, called the
$s_{1,1}$-multiplicity and $u_{1,1}$-number.

In Section~\ref{s:FJ}, we recall two results about the $BT_1$ modules
of Fermat curves and their quotients from \cite{PriesUlmerBT1s}.  In
Section~\ref{s:FEO-general}, we prove the main result about the
Ekedahl--Oort type of $\CC_d$ for all positive integers $d$ relatively
prime to $p$.

In the rest of the paper, we provide explicit examples of the
Ekedahl--Oort structure of $\CC_d$ and the four associated invariants
under various conditions on $d$ and $p$.  In Section~\ref{s:FEO-p=2},
we separate out the case $p=2$.  In Section~\ref{s:anumber}, we
determine the $a$-number of $\CC_d$ for all $d$ relatively prime to
$p$.  In Sections~\ref{s:encompassing} and \ref{s:hermitian}, we
analyze the cases $d=p^a-1$ and $d=p^a+1$, for any natural number $a$.
We call $d=p^a-1$ the ``encompassing case''; it is in some sense the
general case because for every $d'$, the group scheme
${\rm Jac}(C_{d'})[p]$ is a direct factor of ${\rm Jac}(C_d)[p]$ where
$d=p^a-1$ for some $a$.

\subsection*{Acknowledgements} 
Author RP was partially supported by NSF grant DMS-1901819, and
author DU was partially supported by Simons Foundation grants 359573
and 713699.

\section{Groups and modules}\label{s:groups-and-modules}
In this section, we review certain categories
of group schemes and their Dieudonn\'e modules.  

\subsection{Group schemes of $p$-power order
  and their Dieudonn\'e modules}
Our general reference for the assertions in
this section is \cite{Fontaine77}.

Let $W(k)$ denote the Witt vectors over $k$.  Write $\sigma$ for the
absolute Frobenius of $k$, and extend it to $W(k)$ by
$\sigma(a_0,a_1,\dots)=(a_0^p,a_1^p,\dots)$.  Define the
\emph{Dieudonn\'e ring} $\D=W(k)\{F,V\}$ as the $W(k)$-algebra
generated by symbols $F$ and $V$ with relations
\begin{equation} \label{EbasicFV}
  FV=VF=p, \ F\alpha=\sigma(\alpha)F,
  \text{ and } \alpha V=V\sigma(\alpha) \text{ for } \alpha\in W(k).
\end{equation}  Let
$\D_k=\D/p\D\cong k\{F,V\}$.

Let $G$ be a finite commutative group scheme of order $p^\ell$ over
$k$. Let $M(G)$ be the (contravariant) Dieudonn\'e module of $G$.
This is a $W(k)$-module of length $\ell$ with semi-linear
operators $F$ and $V$.  

Let $G^D$ be the Cartier dual of $G$.  For $M$ a $\D$-module of finite
length over $W(k)$, let $M^*$ be its dual module.  A basic result of
Dieudonn\'e theory is that $M(G^D)\cong M(G)^*$.

\subsection{$BT_1$ group schemes and $BT_1$ modules}\label{ss:BT1}

By definition, a \emph{$BT_1$ group scheme} over $k$ is a finite
commutative group scheme $G$ that is killed by $p$ and that has the
properties
\[\Ker(F:G\to G^{(p)})=\im(V:G^{(p)}\to G)\quad\text{and}\quad
  \im(F:G\to G^{(p)})=\Ker(V:G^{(p)}\to G).\]
The notation $BT_1$ is an abbreviation of ``Barsotti--Tate of level 1''
reflecting the fact \cite[Prop.~1.7]{Illusie85} that $BT_1$ group
schemes are precisely those which occur as the kernel of $p$ on a
Barsotti--Tate (=$p$-divisible) group.

By definition, a \emph{$BT_1$ module} over $k$ is a $\D_k$-module $M$
of finite dimension over $k$ such that
\[\Ker(F:M\to M)=\im(V:M\to M)\quad\text{and}\quad
  \im(F:M\to M)=\Ker(V:M\to M).\]
(Oort also calls these $DM_1$ modules.)
Clearly, a $\D_k$-module $M$ is a $BT_1$ module if and only if it is the
Dieudonn\'e module of a $BT_1$ group scheme over $k$.

The group schemes $\Z/p\Z$, $\mu_p$, and $G_{1,1}$ are $BT_1$ group
schemes.  On the other hand, $\alpha_p$ is not, since
$\Ker F=M(\alpha_p)\neq 0=\im V$.

A $BT_1$ group scheme $G$ is
\emph{self-dual} if there exists an isomorphism
$G\cong G^D$.  Similarly, a $BT_1$ module $M$ is \emph{self-dual} if
$M\cong M^*$.  Clearly, $G$ is self-dual if and only if $M(G)$ is
self-dual. 

One may ask that a duality $\phi:G\to G^D$ be skew, meaning that
$\phi^D:G\cong(G^D)^D\to G^D$ satisfies $\phi^D=-\phi$.  This is
equivalent to any of the following three conditions:
\begin{itemize}
\item the bilinear pairing $G\times G\to\G_m$
  induced by $\phi$ is skew-symmetric;
\item the symmetry $M(\phi):M(G)^*\to M(G)$ is skew (meaning
  $M(\phi)^*=-M(\phi)$);
\item the induced bilinear pairing $M(G)^*\times M(G)^*\to k$ is
    skew-symmetric.
\end{itemize}
Interestingly, when $p=2$, there exist $BT_1$ group schemes $G$ with
an \emph{alternating} pairing ($\<x,x\>=0$ for all $x$) such that the
induced pairing on $M(G)$ is skew-symmetric but not alternating.

For this reason, one defines a \emph{polarized $BT_1$ module} over $k$
as a $BT_1$ module over $k$ with a non-degenerate, alternating
pairing, and one defines a \emph{polarized $BT_1$ group scheme} over
$k$ as a $BT_1$ group scheme $G$ over $k$ with a pairing that induces
a non-degenerate, alternating pairing on $M(G)$.
Corollary~\ref{cor:self-dual=>polarized} says that every self-dual
$BT_1$ module admits a polarization which is unique up to
isomorphism. 

If $A$ is a principally polarized abelian variety of dimension $g$
over $k$ , its $p$-torsion subscheme $A[p]$ is naturally a polarized
$BT_1$ group scheme of order $p^{2g}$.

\section{Review of classifications of $BT_1$ group
  schemes}\label{s:BT1} 
In this section, we review bijections between isomorphism classes of
$BT_1$ modules over $k$ and three other classes of objects of
combinatorial nature.  More precisely, following Kraft \cite{Kraft75},
Ekedahl, and Oort \cite{Oort01}, we will construct a diagram:
\begin{equation}\tag{3.1}\label{eq:perms-BT1s}
\xymatrix{
 \text{$BT_1$ modules}\ar@{->}[r]^\sim\ar@{<-}[d]_\sim
    &\text{canonical types}\ar@{->}[d]^\sim\\
\text{\stackanchor{multisets of (primitive)}{cyclic words on
    $\{f,v\}$}}\ar@{<->}[r]
&\text{\stackanchor{(admissible) permutations}{of $S=S_f\cup S_v$.}}
}
\end{equation}

The top horizontal map is the Ekedahl--Oort classification of $BT_1$
modules and the left vertical map is the Kraft classification.  While
some of the material in this section is known, there are significant
reasons to cover it, and we take the opportunity to correct a few
minor imprecisions in the literature.  First, it is helpful to have a
self-contained short description of this material.  Second, we need a
precise dictionary between the classification on the lower right of
the diagram, given in terms of permutations, and the others.  The
permutation classification is not as well known and is particularly
well suited to studying the $p$-torsion group schemes of quotients of
Fermat curves.  Third, the work of Oort uses covariant Dieudonn\'e
theory, but the contravariant theory is more convenient for studying
Fermat curves.

There are other classifications of $BT_1$ modules: one due to Moonen
\cite{Moonen01}, involving cosets of Weyl groups; and another due to 
van der
Geer \cite[\S6]{vanderGeer99} in terms of Young diagrams and partitions.
We will not need this material, so we omit any further discussion.


\subsection{Words and permutations}\label{ss:words-perms}
Let $\WW$ be the monoid of words $w$ on the two-letter alphabet
$\{f,v\}$, and write $1$ for the empty word.  By  convention, the
first (resp.\ last) letter of $w$ is its leftmost (resp.\ rightmost)
letter.  The complement $w^c$ of $w$ is the word obtained by
exchanging $f$ and $v$ at every letter.

For a positive integer $\lambda$, write $\WW_\lambda$ for the words of
length $\lambda$.  Endow $\WW_\lambda$ with the lexicographic ordering
with $f<v$.
If $w \in \WW_\lambda$, we write $w=u_{\lambda-1}\cdots u_0$ where
$u_i\in\{f,v\}$ for $0 \leq i \leq \lambda-1$.  Define an action of
the group $\Z$ on $\WW$ by requiring that $1\in\Z$ map
$w=u_{\lambda-1}\cdots u_0$ to $u_0u_{\lambda-1}\cdots u_1$.  If $w$
and $w'$ are in the same orbit of this action, we say $w'$ is a
\emph{rotation} of $w$.  The orbit $\overline{w}$ of $w$ under the
action of $\Z$ is called a \emph{cyclic word}.  Write $\overline\WW$
for the set of cyclic words.

A word $w$ is \emph{primitive} if $w$ is not of the from $(w')^e$ for
some word $w'$ and some integer $e>1$.  If $w$ has length $\lambda>0$,
it is primitive if and only if the subgroup of $\Z$ fixing $w$ is
exactly $\lambda\Z$.  Write $\WW'$ for the set of primitive words 
and $\overline\WW'$ for the set of primitive cyclic words.
At the level of multisets, one can define a retraction of
$\overline\WW'\subset\overline\WW$ by sending the class of a word
$w=(w')^e$ where $w'$ is primitive to the class of $w'$ with
multiplicity $e$.





Consider a finite set $S$ written as the disjoint union
$S=S_f\cup S_v$ of two subsets and a permutation $\pi:S\to S$.  
Two such collections of data $(S=S_f\cup S_v,\pi)$ and
$(S'=S'_f\cup S'_v,\pi')$ are isomorphic if there is a bijection
$\iota: S\to S'$ such that $\iota(S_f)=S_f'$, $\iota(S_v)=S_v'$, and
$\iota \pi = \pi' \iota$.

Given $(S=S_f\cup S_v,\pi)$, there is an associated multiset of cyclic
words on $\{f,v\}$ defined as follows.  For $a\in S$ with orbit of
size $\lambda$, define the word $w_a=u_{\lambda-1}\cdots u_0$ where
\[u_j = f \text{ if } \pi^j(a)\in S_f, \text{ and } u_j = v \text{ if
  } \pi^j(a)\in S_v.\] Then $\overline{w}_a$ depends only on the orbit
of $a$.  This gives a well-defined map from orbits of $\pi$ to cyclic
words, i.e., elements of $\overline\WW$.  Taking the union over
orbits, we can associate to $(S=S_f\cup S_v,\pi)$ a multiset of cyclic
words.  If $S$ and $S'$ are isomorphic, then they yield the same
multiset.

Conversely, given a multiset of cyclic words, let $S$ be the set of
all words representing them (repeated to account for multiplicities),
let $S_f$ be the subset of those words ending with $f$, let $S_v$ be
the subset of those words ending with $v$, and let $\pi$ be defined by
the action of $1\in\Z$ as above.

The data $(S=S_f\cup S_v,\pi)$ is
\emph{admissible} if the associated words $w_a$ for $a\in S$ are
primitive.  

\subsection{Cyclic words to $BT_1$ modules}\label{ss:words-to-BT1s}
Following Kraft \cite{Kraft75}, we attach a $BT_1$ module to a
multiset of primitive cyclic words.  This defines the left vertical arrow
of Diagram~(3.1).

\subsubsection{Construction}
Suppose that $w\in\WW'$ is a primitive word, say
$w=u_{\lambda-1}\cdots u_0$ with $u_j\in\{f,v\}$.  Let $M(w)$ be the
$k$-vector space with basis $e_j$ with $j\in\Z/\lambda\Z$ and define a
$p$-linear map $F:M(w)\to M(w)$ and a $p^{-1}$-linear map
$V:M(w)\to M(w)$ by setting
\[F(e_{j})=\begin{cases}
    e_{{j+1}}&\text{if $u_j=f$},\\
    0&\text{if $u_j=v$,}
  \end{cases}
\qquad\text{and}\qquad V(e_{{j+1}})=
  \begin{cases}
    e_{{j}}&\text{if $u_j=v$},\\
    0&\text{if $u_j=f$.}  \end{cases}\]
(Note that $F(e_j)\neq0$ exactly when the last letter of the
$j$-th rotation of $w$ is $f$,
and $V(e_{j+1})\neq0$ exactly when the last letter of the $j$-th
rotation of $w$ is $v$.)
This construction yields a $BT_1$ module which up to
isomorphism only depends on the primitive cyclic word $\overline{w}$
associated to $w$.

Kraft proves that $M(w)$ is indecomposable and that every
indecomposable $BT_1$ module is isomorphic to one of the form $M(w)$
for a unique primitive cyclic word $\overline{w}$.  Thus every $BT_1$
module $M$ is isomorphic to a direct sum $\oplus M(w_i)$ where
$\overline{w}_i$ runs through a uniquely determined multiset of
primitive cyclic words.

If $w$ is a word that is not necessarily primitive, the formulas above
define a $BT_1$ module.  If $w=(w')^e$, Kraft also proves that
$M(w)\cong M(w')^e$.

It is clear that $M(f)=M(\Z/p\Z)$, $M(v)=M(\mu_p)$, and $M(fv)$ is the
Dieudonn\'e module of the kernel of $p$ on a supersingular elliptic
curve.  More generally, if $w$ has length $>1$ and is primitive, then
$M(w)$ is the Dieudonn\'e module of a unipotent, connected $BT_1$
group scheme.

\subsubsection{Generators and relations}
Let $w$ be a primitive word with associated $BT_1$ module $M(w)$.  It
will be convenient to have a presentation of $M(w)$ by generators and
relations.  Clearly, $M(f) =\D_k/(F-1,V)$ and 
$M(v) =\D_k/(F,V-1)$.

Now suppose $w$ has length $>1$.  Then, after rotating $w$ if
necessary, we may assume its last letter is $f$ and its first
letter is $v$.  (Both letters appear because $w$ is primitive, so in
particular is not $f^m$ nor $v^n$.)  We then write $w$ in exponential
notation as 
\begin{equation*}
w=v^{n_r}f^{m_r}\cdots v^{n_1}f^{m_1},
\end{equation*}
for some positive integers $r,m_1,\dots,m_{r},n_1,\dots,n_r$.

\begin{lemma} \label{Lrelations}
  The $BT_1$ module $M(w)$ admits generators $E_i$ for $i\in\Z/r\Z$
  with relations $F^{m_i}E_{i-1} = V^{n_i}E_i$ for $i=1,\dots,r$.
\end{lemma}

\begin{proof}
  Indeed, for $i=0,\dots, r-1$, let $I(i)=\sum_{j=1}^i (m_j+n_j)$,
  let $I'(i)=I(i)+m_{i+1}$, and let $E_i=e_{I(i)}$.  The $E_i$
  generate $M(w)$ as a $\D_k$-module because
\[\text{if }I(i)\le j\le I'(i),\text{ then }e_j=F^{j-I(i)} E_i,\]
\[\text{and if }I'(i)\le j\le I(i+1),\text{ then }e_j=V^{j-I'(i)} E_{i+1},\]
and there are relations $F^{m_i}E_{i-1}=e_{I'(i-1)}=V^{n_i}E_i$.
  \end{proof}

The following diagram illustrates this presentation of $M(w)$:
  \[\xymatrix{
      E_{r-1}=e_{I(r-1)}\ar@{|->}[dr]_{F^{m_r}}&&
      E_0=e_{I(0)}\ar@{|->}[dl]^{V^{n_r}}\ar@{|->}[dr]_{F^{m_1}}&&
      E_1=e_{I(1)}\ar@{|->}[dl]^{V^{n_1}}\ar@{|->}[dr]_{F^{m_2}}&&
      \cdots\ar@{|->}[dl]^{V^{n_2}}\\
&e_{I'(r-1)}&&e_{I'(0)}&&e_{I'(1)}}\]

\subsection{$BT_1$ modules to canonical types}\label{ss:BT1s-to-types}
Following Oort \cite{Oort01}, we explain how to
describe the isomorphism class of a $BT_1$ module in terms of certain
combinatorial data.  This defines the top horizontal arrow in
Diagram~(3.1). Readers are invited to work through the example in
Section~\ref{sss:example} while reading this section.

Warning: Many of our formulas differ from those in \cite{Oort01} for
two reasons: first, we use the contravariant Dieudonn\'e theory,
whereas Oort uses the covariant theory; second, Oort studies a
filtration defined by $F^{-1}$ and $V$, whereas we use $F$ and
$V^{-1}$.  The two approaches are equivalent (and exchanged under
duality), but the latter is more convenient for studying
Fermat curves.

\subsubsection{The canonical filtration}
Recall that $\WW$ denotes the monoid of words on 
$\{f,v\}$.  Let $M$ be a $BT_1$ module, and define a left action of
$\WW$ on the set of $k$-subspaces of $M$ by requiring that
\[fN :=F(N) \text{ and } vN:=V^{-1}(N).\]
In other words, $f$ sends a subspace $N$ to its image under $F$ and
$v$ sends $N$ to its inverse image under $V$.  Note that if
$N_1\subset N_2$, then $fN_1\subset fN_2$ and $vN_1\subset vN_2$.
If $N$ is a $\D_k$-module, so are $fN$ and $vN$, and $fN\subset
N\subset vN$.  Also note that $fM = \im F =\Ker V = v0$. 

Let $M$ be a $BT_1$ module.  An \emph{admissible filtration} on $M$ is
a filtration by $\D_k$-modules
\begin{equation} \label{Ecoarse}
0=M_0\subsetneq M_1\subsetneq\cdots\subsetneq M_s=M,
\end{equation}
such that for all $i$, there exist indices $\phi(i)$ and $\nu(i)$
such that $fM_i=M_{\phi(i)}$ and $vM_i=M_{\nu(i)}$.  

\begin{defn} \label{Dcoarse}
The \emph{canonical filtration} on $M$ is the coarsest admissible
  filtration on $M$.  If Equation~\eqref{Ecoarse} is the
  canonical filtration, the \emph{blocks} of $M$ are $B_i=M_{i+1}/M_i$
  for $0\le i\le s-1$.
\end{defn}




The canonical filtration of $M$ is constructed by enumerating all
subspaces of $M$ of the form $wM$ and indexing them in order of
containment.  Define $s$ to be the number of steps in the filtration
and $r$ to be the integer such that $M_r=fM=v0$.  Define functions
\[\phi:\{0,\dots,s\}\to\{0,\dots,r\},\quad
  \nu:\{0,\dots,s\}\to\{r,\dots,s\},\quad \rho:\{0,\dots,s\}\to\Z\]
by $fM_i=M_{\phi(i)}$, $vM_i=M_{\nu(i)}$, and $\rho(i)=\dim_k M_i$. 
This data has the following properties.

\begin{prop-def} \label{Dcanonicaltype}
  The data $(r,s,\phi,\nu,\rho)$ associated to the canonical
  filtration of $M$ is a \emph{canonical type}, i.e., $s>0$,
  $0\le r\le s$, and the functions $\phi$, $\nu$, and $\rho$ have the
  following properties:
\begin{enumerate} 
\item $\phi$ and $\nu$ are monotone nondecreasing and surjective;
\item $\rho$ is strictly increasing with $\rho(0)=0$;
\item\label{eq:can1} $\nu(i+1)>\nu(i)$ if and only if $\phi(i+1)=\phi(i)$;
\item if the equivalent conditions in \eqref{eq:can1} are true, then
$\rho(i+1)-\rho(i)=\rho(\nu(i)+1)-\rho(\nu(i))$,
while if not, then  $\rho(i+1)-\rho(i)=\rho(\phi(i)+1)-\rho(\phi(i))$;
\item and every integer in $\{1,\dots,s\}$ can be obtained by repeatedly
  applying $\phi$ and $\nu$ to $s$.
\end{enumerate}
\end{prop-def}


If the data $(r,s,\phi,\nu,\rho)$ comes from the canonical filtration of a
$BT_1$ module, then it is clear from the definitions that 
$0\le r\le s$, with $s>0$, and that properties (1), (2), and (5) hold.  Oort
proves that properties (3) and (4) hold in \cite[\S2]{Oort01}.

The properties imply that $\nu(i+1)-\nu(i)$ and $\phi(i+1)-\phi(i)$
are either 0 or 1, that exactly one of them is 1, and that
$\nu(i)+\phi(i)=r+i$.  The next lemma will be used in later sections.

\begin{lemma} \label{rem:subquotients} \cite[Lemma~2.4]{Oort01} 
Let \eqref{Ecoarse} denote the canonical filtration of $M$ and let 
$B_i=M_{i+1}/M_i$ for $0\le i\le s-1$.  If $\phi(i+1)>\phi(i)$ then
$F$ induces a $p$-linear isomorphism $B_i\isoto B_{\phi(i)}$, and if
$\nu(i+1)>\nu(i)$, then $V^{-1}$ induces a $p$-linear isomorphism
$B_i\isoto B_{\nu(i)}$.  
\end{lemma}

The key assertion is that the canonical type of $M$ determines $M$ up
to isomorphism:
\begin{propss}\label{prop:module->CT}
  If the canonical types of two $BT_1$ modules $M, M'$ are equal, then
  $M\cong M'$.
\end{propss}

Oort proves a related result \cite[Thm.~9.4]{Oort01} involving
quasi-polarizations (pairings) which is more involved and only applies
to self-dual $BT_1$ modules.  Moonen proves the result stated here
\cite[\S4]{Moonen01} in the more general context where the module $M$
also has endomorphisms by a semi-simple $\Fp$-algebra $D$; taking
$D=\Fp$ yields Proposition~\ref{prop:module->CT}.

\begin{remss}
 Let $\mu:\{0,\dots,s-1\}\to\Z$ be defined by $\mu(i)=\rho(i+1)-\rho(i)$.
  Equivalently, $\mu(i)=\dim_k(B_i)$.  Property (2) says $\mu$ takes
  positive values, and property (4) says if $\nu(i+1)>\nu(i)$, then
  $\mu(i)=\mu(\nu(i))$ and if $\phi(i+1)>\phi(i)$, then
  $\mu(i)=\mu(\phi(i))$.
\end{remss}

\begin{remss}
  Oort defines a canonical type to be data as above satisfying
  properties (1)--(4), i.e., he omits (5), and he states
  \cite[Remark~2.8]{Oort01} that every canonical type comes from a
  $BT_1$ module.  With this definition, it is true that every
  canonical type comes from an admissible filtration on a $BT_1$
  module, but not necessarily from the \emph{canonical} filtration.
  Here is a counterexample: Let $r=0$, $s=2$ and $\phi(i)=0$,
  $\nu(i)=i$ and $\rho(i)=i$ for $i=0,1,2$.  This data comes from a
  filtration on 
  $N =M((\mu_2)^2)$
  and it satisfies properties (1)--(4), but
  not (5).  The canonical type of $N$ 
  has $r=0$, $s=1$, $\phi(i)=0$, $\nu(i)=i$ for $i=0,1$, and
  $\rho(0)=0$, $\rho(1)=2$.
\end{remss}

\subsubsection{An example}\label{sss:example}
Let $M$ be the $k$-vector space with basis $e_1,\dots,e_7$ and action
of $F, V$ given by
\begin{center}
\begin{tabular}{| c | c | c | c | c | c | c | c | }
\hline
$e$ & $e_1$ & $e_2$ & $e_3$ & $e_4$ & $e_5$ & $e_6$ & $e_7$ \\ 
\hline
$F(e)$ & 0 & 0 & 0 & $e_1$ & $e_2$ & 0 & $e_3$ \\
\hline
$V(e)$ & 0 & 0 & 0 & $e_1$ & $e_2$ & $e_3$ & $e_6$ \\
\hline 
\end{tabular}.
\end{center}
Using $\<\dots\>$ to denote the span of a set of vectors, the
canonical filtration of $M$ is 
\begin{multline*}
M_0=0\subset M_1=\<e_1,e_2\>\subset M_2=\<e_1,e_2,e_3\>
\subset M_3=\<e_1,e_2,e_3,e_4,e_5\>\\
\subset M_4=\<e_1,e_2,e_3,e_4,e_5,e_6\>
  \subset M_5=\<e_1,e_2,e_3,e_4,e_5,e_6,e_7\>=M.
\end{multline*}
The canonical type is given by $s=5$, $r=2$, and the functions $\phi,
\nu, \rho$ below: 
\begin{center}
\begin{tabular}{| c | c | c | c | c | c | c |}
\hline
$i$ & 0 & 1 & 2 & 3 & 4 & 5 \\ \hline
$\phi(i)$ & 0 & 0 & 0 & 1 & 1 & 2 \\ \hline
$\nu(i)$  & 2 & 3 & 4 & 4& 5 & 5 \\ \hline
$\rho(i)$ & 0 & 2 & 3 & 5 & 6 & 7 \\
  \hline
\end{tabular}.
\end{center}

\subsection{Canonical types to permutations}\label{ss:types-to-perms}

Following \cite[Section 2]{Oort01}, we explain how to use
a canonical type to define a partitioned set $S=S_f\cup S_v$ with
permutation $\pi:S\to S$, thus defining the right vertical arrow
of Diagram~(3.1).  These results cast considerable light
on the structure of $BT_1$ modules, but they will not be used explicitly
in the rest of the paper.

Let $(r,s,\phi,\nu,\rho)$ be as in
Definition~\ref{Dcanonicaltype}, and let $\Gamma=\{0,\dots,s-1\}$.
Define $\Pi:\Gamma\to\Gamma$ by:
\[\Pi(i)=\begin{cases}
    \phi(i)&\text{if $\phi(i+1)>\phi(i)$,}\\
    \nu(i)&\text{if $\nu(i+1)>\nu(i)$.}
  \end{cases}\]
Property~(3) of Definition~\ref{Dcanonicaltype} shows that $\Pi$ is
well defined.  Property~(1)
implies that $\Pi$ is injective and thus bijective.  By Property~(4), 
$\mu(i)=\rho(i+1)-\rho(i)$ is constant on the orbits of $\Pi$.

We partition $\Gamma$ as a disjoint union $\Gamma_f\cup\Gamma_v$ where
$i\in\Gamma_f$ if and only if $\pi(i)=\phi(i)$.  Equivalently,
\[\Gamma_f=\{i\in\Gamma\mid \phi(i+1)>\phi(i)\}
  \quad\text{and}\quad
  \Gamma_v=\{i\in\Gamma\mid\nu(i+1)>\nu(i)\}.\]
Then $(r,s,\phi,\nu,\rho)$ is
determined by the data $\Gamma=\Gamma_f\cup\Gamma_v$,
$\Pi:\Gamma\to\Gamma$, and $\mu:\Gamma\to\Z$.

\subsubsection{Example \ref{sss:example} continued}
In this case, the permutation $\Pi$ of $\Gamma=\{0,1,2,3,4\}$ is
$(0,2)(1,3,4)$, the partition is given by $\Gamma_f=\{2,4\}$ and
$\Gamma_v=\{0,1,3\}$, and the associated words are
\[w_0=fv, \ w_1=fvv, \ w_2=vf, \ w_3=vfv, \ \text{and }w_4=vvf.\] Note
that $\mu(0)=\mu(2)=2$ and $\mu(1)=\mu(3)=\mu(4)$, so $\mu$ is
constant on the orbits of $\Pi$.
\medskip

To complete the definition of the right vertical arrow of
Diagram~(3.1), we use $\mu$ as a set of ``multiplicities'' to expand
$\Gamma$ into $S$.  More precisely, define
\[S:=\left\{\left.e_{i,j}\,\right|\,i\in\Gamma, 1\le j\le \mu(i)\right\}\]
with partition
\[S_f=\left\{\left.e_{i,j}\in S\,\right|\, i\in\Gamma_f\right\}
  \quad\text{and}\quad
S_v=\left\{\left.e_{i,j}\in S\,\right|\, i\in\Gamma_v\right\}\]
and permutation $\pi:S\to S\quad\text{with}\quad \pi(e_{i,j}):=e_{\Pi(i),j}$.

\begin{lemma}
The data $(S=S_f\cup S_v,\pi)$ is an \emph{admissible}
permutation. 
\end{lemma}

\begin{proof}
  The set of cyclic words associated to $(S=S_f\cup S_v,\pi)$ is the
  same as the set of cyclic words associated to
  $(\Gamma=\Gamma_f\cup\Gamma_v,\Pi)$.  We need to show that these
  words are primitive.  If $w=w_i$ for $i \in \Gamma$, then
  $w_{\Pi^j(i)}$ is the $j$-th cyclic rotation of $w_i$.  Thus to show
  that the $w_i$ are all primitive, it suffices to show that they are
  distinct.


  To that end, define a left action of the monoid $\WW$ on the set
  $\{0,\dots,s\}$ by requiring that $f(i)=\phi(i)$ and $v(i)=\nu(i)$.
  If $i\in\Gamma$, and if $w_i$ is the word associated to $i$, then
  $w_i$ fixes $i$ and $i+1$.  This is a manifestation in the canonical
  type of the isomorphisms from Lemma~\ref{rem:subquotients}:
\[B_i\isoto B_{\Pi(i)}\isoto B_{\Pi^2(i)}\isoto \cdots\isoto B_i.\]

Now assume that $i,j\in\Gamma$, $i<j$, and $w_i=w_j=w$.  We will
deduce a contradiction of property~(5) in
Definition~\ref{Dcanonicaltype}.  Since $\phi$ and $\nu$ are
nondecreasing, for all $n>0$ we have
\[w^n(s)\ge i+1>i\ge w^n(0)
\quad\text{and}\quad
w^n(s)\ge j+1>j\ge w^n(0),\]
so $w^n(s)\ge j+1$ and $i\ge w^n(0)$ for all $n>0$.

Choose some $i'$ with $i<i'\le j$.  By property (5) of
Definition~\ref{Dcanonicaltype}, there is a word $w'$ with $w'(s)=i'$.
Choose $n>0$ large enough that $w^n$ is at least as long as $w'$, and
then replace $w'$ with $w'v^m$ where $m$ is chosen so $w^n$ and $w'$
have the same length.  Since $v(s)=\nu(s)=s$, we still have
$w'(s)=i'$.
If $w'\ge w^n$
(in the lexicographic order from Section~\ref{ss:words-perms}), then
$i'=w'(s)\ge w^n(s)\ge j+1$, a contradiction; and if $w'<w^n$, then
$i'=w'(s)\le w^n(0)\le i$, again a contradiction.  We conclude that
there can be no $i<j$ with $w_i=w_j$, and thus $\pi$ is admissible.
\end{proof}

\begin{remss}
  A more thorough analysis along these lines shows that if $M$ is a
  $BT_1$ module, then there are finitely many primitive words $w_i$
  such that $w_i^nM\supsetneq w_i^n0$ for all $n>1$.  Enumerate these
  as $w_0,\dots,w_{s-1}$ and choose integers $n_i$ so that
  $w_i^nM=w_i^{n_i}M$ and $w_i^{n}0=w_i^{n_i}0$ for all $n\ge n_i$, and
  so that the lengths of the $w_i^{n_i}$ are all the same.  Define $\tilde
  w_i=w_i^{n_i}$.
  Reorder the $w_i$ so that
  \[ \tilde w_0< \tilde w_1<\cdots< \tilde w_{s-1}.\]
  (Numbering the $\tilde\omega_i$ from $i=0$ turns out to be most
  convenient; see the proof of Lemma~\ref{lemma:k}.)
  Let $\tilde w_{-1}M=0$.
  Then the $w_i$ are distinct, the canonical filtration of $M$ is
  \[0\subsetneq \tilde w_0M\subsetneq \cdots\subsetneq \tilde w_{s-1}M=M,\]
  and the primitive words associated to $M$ are precisely the $w_i$,
  with the multiplicity of $w_i$ being
   \[\mu(i)=\dim_k(\tilde w_{i}M/
    \tilde w_{i-1}M).\]
\end{remss}

\subsection{Words to canonical types}\label{ss:words-to-canonical}
In this section, we describe the map from
multisets of (not necessarily primitive) cyclic words to canonical
type.  This will be used in Sections~\ref{s:Homs} and
\ref{s:FEO-general}.

For $1\le i\le n$, let $\overline w_i$ be cyclic words with
multiplicities $m_i$.  Let $M(\overline w_i)$ be the Kraft module discussed in
Section~\ref{ss:words-to-BT1s}.  Our goal is to describe the canonical type of
the $BT_1$ module
\[M=\oplus_{i=1}^n M(\overline w_i)^{m_i}.\]

Let $\lambda_i$ be the length of $\overline w_i$ and choose a
representative $w_i=u_{i,\lambda_i-1}\cdots u_{i,0}$ of
$\overline w_i$ with $u_{i,j}\in\{f,v\}$.  The $k$-vector space
underlying $M$ has basis $e_{i,j,k}$ where $1\le i\le n$,
$ j\in\Z/\lambda_i\Z$, and $1\le k\le m_i$.  Its $\D_k$-module
structure is given by
\[F(e_{i,j,k})=\begin{cases}
    e_{i,{j+1},k}&\text{if $u_{i,j}=f$},\\
    0&\text{if $u_{i,j}=v$,}
  \end{cases}
\qquad\text{and}\qquad V(e_{i,{j+1},k})=
  \begin{cases}
    e_{i,j,k}&\text{if $u_{i,j}=v$},\\
    0&\text{if $u_{i,j}=f$.}  \end{cases}\]

Let $w_{i,j}=u_{i,j-1}\cdots u_{i,0}u_{i,\lambda_i}\cdots u_{i,j}$ be
the $j$-th rotation of $w_i$.  Let
$\ell={\rm LCM}(\lambda_i)_{i=1..n}$ and set
$\tilde w_{i,j}=w_{i,j}^{\ell/\lambda_i}$.  Thus each $w_{i,j}$ has
length $\ell$ and we may compare them in lexicographic order.

Let $\Sigma$ be the multiset obtained by including each
$\tilde w_{i,j}$ with multiplicity $m_i$.  (The $\tilde w_{i,j}$ need
not be distinct, and if there are repetitions, one should add
multiplicities.)  Now relabel the distinct elements of $\Sigma$ as
$\omega_t$ for $0\le t\le s-1$ and ordered so that
$\omega_0<\omega_1<\cdots<\omega_{s-1}$.  Let $\tau$ be the function
such that $\tau(i,j)=t$ if and only if $\tilde w_{i,j}=\omega_t$.  Let
$\mu(t)$ be the multiplicity of $\omega_{t}$ in $\Sigma$.

For $1\le t\le s$, let $M_t$ be the $k$-subspace of $M$ spanned
by those $e_{i,j,k}$ with $\tau(i,j)\le t-1$.  We claim that the
canonical filtration of $M$ is
\[0=M_0\subsetneq M_1\subsetneq\cdots\subsetneq M_{s}=M.\] Indeed, it
is easy to see that if $\omega_t$ ends with $f$ and
$\omega_{t'}$ is the first rotation of $\omega_t$, then $F$
induces a $p$-linear isomorphism
$M_{t+1}/M_{t}\isoto M_{t'+1}/M_{t'}$.  On the other hand,
if $\omega_t$ ends with $v$, then $V^{-1}$ induces a $p$-linear
isomorphism $M_{t+1}/M_{t}\isoto M_{t'+1}/M_{t'}$.  Thus
the displayed filtration is an admissible filtration.  Via the action
of $\WW$ on $M$, the word $\omega_t$ induces a semi-linear
automorphism of $M_{t+1}/M_{t}$, while a power of $\omega_t$
induces the zero map of $M_{t'+1}/M_{t'}$ if $t\neq t$.
It follows that $\omega^n_{t} M=M_{t+1}$ for large enough $n$,
so this is the coarsest filtration, thus the canonical filtration.

The dimension of the block $B_t=M_{t+1}/M_t$ equals
$\mu(\omega_t)$, the multiplicity of $\omega_t$ in
$\Sigma$.

It remains to record the values of $r$ and the functions $\phi$,
$\nu$, and $\rho$ associated to $M$:  \begin{itemize}
\item $r=\#\left\{t \mid 0\le t<s,\ \omega_t
      \text{ ends with }f \right\}$;
\item $\phi(i)=\#\left\{t \mid 0\le t< i,\ \omega_t
      \text{ ends with }f \right\}$;
\item $\nu(i)=r+\#\left\{t \mid 1\le t< i,\ \omega_t
      \text{ ends with }v \right\}$; and
    \item $\rho(i)=\sum_{t=0}^{i-1}\mu(t)$.
\end{itemize}

\subsubsection{Example}
Let $\overline w_1=\overline{fv}$, $w_2=\overline{fvfv}$, and
$m_1=m_2=1$.  Taking $n_1=2$ and $n_2=1$, one finds that $T$ contains
$\omega_0=fvfv$ and $\omega_1=vfvf$, each with multiplicity 3.  The
function $\tau$ is
\[\tau(1,0)=\tau(2,0)=\tau(2,2)=0
  \quad\text{and}\quad
  \tau(1,1)=\tau(2,1)=\tau(2,1)=1,\]
and $\mu(0)=\mu(1)=3$.
Thus $s=2$, $r=1$, and the functions $\phi$, $\nu$,
and $\rho$ are given by
\begin{center}

\begin{tabular}{| c | c | c | c | }
\hline
$i$ & 0 & 1 & 2 \\ \hline
$\phi(i)$ & 0 & 0 & 1 \\ \hline
$\nu(i)$  & 1 & 2 & 2\\ \hline
$\rho(i)$ & 0 & 3 & 6 \\ \hline
\end{tabular}
\end{center}

\section{Duality and E--O structures} \label{s:refinements}

\subsection{Duality of $BT_1$ modules}\label{ss:duality}
We record how duality of $BT_1$ modules interacts with
the objects in Diagram~(3.1).  All the assertions in this section will
be left to the reader.

For a $BT_1$ module $M$, let $M^*$ is its dual.  If $N\subset M$ is a
$k$-subspace, then
\[F\left(N^\perp\right)=\left(V^{-1}N\right)^\perp
  \quad\text{and}\quad
V^{-1}\left(N^\perp\right)=\left(FN\right)^\perp.\]
Let $0=M_0\subsetneq M_1\subsetneq\cdots\subsetneq M_s=M$
be the canonical filtration of $M$; setting $M^*_i=(M_{s-i})^\perp$,
then the canonical filtration of $M^*$ is
$0=M^*_0\subsetneq M^*_1\subsetneq\cdots\subsetneq M^*_s=M^*$.
If the canonical data attached to $M$ is $(r,s,\phi,\nu,\rho)$, and
the canonical data attached to $M^*$ is
$(r^*,s^*,\phi^*,\nu^*,\rho^*)$, then $s^*=s$, $r^*=s-r$, and for
$0 \leq i \leq s$,
\begin{equation*}
  \phi^*(i)=s-\nu(s-i), \ \nu^*(i)=s-\phi(s-i),
  \text{ and } \rho^*(i)=\rho(s)-\rho(s-i).
\end{equation*}
It follows that $M$ is self-dual if and only if the associated
canonical data satisfies $s=2r$,
\begin{equation}\label{eq:dual-canonical}
  \phi(i)+\nu(s-i)=s, \text{ and } \rho(i)+\rho(s-i)=\rho(s).
\end{equation}
The relationship between the partitioned set with permutation
associated to $M$ and to $M^*$ is $S^*=S$, $S^*_f=S_v$, $S^*_v=S_f$,
and $\pi^*=\pi$.  It follows that $M$ is self-dual if and only if
there exists a bijection $\iota:S\isoto S$ which satisfies $\iota(S_f)=S_v$
and $\pi\circ\iota=\iota\circ \pi$.

If $w$ is a primitive word on $\{f,v\}$, define $w^c$ to be
the word obtained by exchanging $f$ and $v$.  This operation descends
to a well-defined involution on cyclic words and $M(w)^*\cong M(w^c)$.
It follows that $M$ is self-dual if and only if the associated
multiset of cyclic primitive words consists of self-dual words
($\overline w^c=\overline w$) and pairs of dual words
($\{\overline w,\overline w^c\}$).

\subsection{Ekedahl--Oort classification of polarized $BT_1$
  modules}\label{ss:EO-structures} 
Clearly, a polarized $BT_1$ module is self-dual.  Conversely, as we
will see below (Corollary~\ref{cor:self-dual=>polarized}), any
self-dual $BT_1$ module can be given a polarization.
In this section, we review the Ekedahl--Oort classification
\cite{Oort01} of polarized $BT_1$ modules.

\subsubsection{Elementary sequences} \label{Selemseq}
Elementary sequences are a convenient repackaging of the data of a
self-dual canonical type $(r,s,\phi,\nu,\rho)$.  Using
Equation~\eqref{eq:dual-canonical}, the restrictions of $\phi$ and
$\rho$ to $\{0,\dots,r\}$ determine the rest of the data.  An
\emph{elementary sequence} of length $g$ is a sequence
$\Psi= [\psi_1,\dots,\psi_g]$ of integers with
$\psi_{i-1}\le\psi_i\le\psi_{i-1}+1$ for $i=1,\dots,g$.
The set of elementary sequences of length $g$ has cardinality $2^g$.

Given $(r,s,\phi,\nu,\rho)$, define an elementary sequence as
follows.
Let $g=\rho(r)$.  Set $\psi_0=0$. 
For each $1\le j\le g$, let $i$ be the unique integer $0<i\le r$ such
that $\rho(i-1)<j\le\rho(i)$.  Define
\[\psi_j=
  \begin{cases}
    \psi_{j-1}&\text{if $\phi(i)=\phi(i-1)$,}\\
    \psi_{j-1}+1&\text{if $\phi(i)>\phi(i-1)$.}
  \end{cases}\]
(Put more vividly, the sequence $\psi_j$ increases for $\mu(i-1)$
steps if $\phi(i)>\phi(i-1)$ and it stays constant for $\mu(i-1)$ steps
if $\phi(i)=\phi(i-1)$.)

We leave it as an exercise for the reader to check that, given an
elementary sequence, there is a unique self-dual canonical type giving
rise to it by this construction.

Elementary sequences can be obtained directly from a self-dual $BT_1$
module as follows: The canonical filtration can be refined into a
``final filtration,'' i.e., a filtration
\[0=M_0\subsetneq M_1\subsetneq\cdots\subsetneq M_{2g}=M\] respected
by $F$ and $V^{-1}$ and such that $\dim_k(M_i)=i$.  Then $\Psi$ is
defined by $\psi_i=\dim_k(FM_i)$.


\begin{thmss}\label{thm:Oort-classification} \cite[Thm.~9.4]{Oort01}
  Every elementary sequence of length $g$ arises from a polarized
  $BT_1$ module of dimension $2g$, and two polarized $BT_1$ modules
  over $k$ with the same elementary sequences are isomorphic.  More
  precisely, there is an isomorphism of $BT_1$ modules which respects
  the alternating pairings.  \textup{(}This isomorphism is not unique
  in general.\textup{)}
\end{thmss}

The elementary sequence attached to a $BT_1$ module is also called its
\emph{Ekedahl--Oort structure}.

\begin{corss}\label{cor:self-dual=>polarized}
  Every self-dual $BT_1$ module admits a polarization, i.e., a
  non-degenerate \emph{alternating} pairing, and this pairing is
  unique up to \textup{(}non-unique\textup{)} isomorphism.
\end{corss}

\begin{proof}
  If $M$ is a self-dual $BT_1$ module, construct its canonical type,
  and its elementary sequence $\Psi$ as in Section~\ref{Selemseq}.
  Theorem~\ref{thm:Oort-classification} furnishes a polarized $BT_1$
  module with the same underlying $BT_1$ module, and this proves the
  existence of a polarization.  For uniqueness, note that the
  construction of $\Psi$ does not depend on the
  pairing.  So, given two alternating pairings on $M$,
  Theorem~\ref{thm:Oort-classification} shows there is a (not
  necessarily unique) automorphism intertwining the
  pairings.
\end{proof}

\section{Homomorphisms}\label{s:Homs}
The Kraft description of $BT_1$ modules is well adapted to computing
homomorphisms.   We work out three important examples in this section.

\subsection{Homs from $\Z/p\Z$ or $\mu_p$}\label{ss:homs-Z/pZ}

\begin{defn}
If $G$ is a $BT_1$ group scheme, the \emph{$p$-rank} of $G$
is the largest integer $f$ such that there is an injection
$(\Z/p\Z)^f\into G$.  Alternatively, $f$ is the dimension of
the largest quotient space of $M(G)$ on which Frobenius acts
bijectively.
\end{defn}

\begin{lemma} \label{Lprank}
  If $G$ is a $BT_1$ group scheme, the $p$-rank of $G$ is equal to the
  multiplicity of the word $f$ in the multiset of primitive words
  corresponding to $M(G)$.  Similarly, the largest $f$ such that
  $\mu_p^f$ embeds in $M(G)$ is the multiplicity of the word $v$.
\end{lemma}

\begin{proof}
  From the presentation in terms of generators and relations, we see
  that if $w$ is a primitive word other than $f$, then there is no
  non-zero homomorphism from $\Z/p\Z=M(f)$ to $M(w)$.  It follows that
  if $M(G)$ is given in the Kraft classification by a multiset of
  primitive cyclic words, the $p$-rank is the multiplicity of the word
  $f$.  The assertion for $\mu_p$ is proved analogously.
\end{proof}

\subsection{Homs from $\alpha_p$}\label{ss:homs-ap}
Let $G$ be a finite group scheme over $k$ killed by $p$.  Define the
\emph{foot} (or \emph{socle}) of $G$ to be the largest semisimple
subgroup of $G$.  The simple objects in the category of finite group
schemes over $k$ killed by $p$ are $\Z/p\Z$, $\mu_p$, and $\alpha_p$.
Thus the foot of $G$ is a direct sum of these (with multiplicity).  If
$G$ is connected and unipotent, then its foot is of the form
$\alpha_p^\ell$ for some positive integer $\ell$.
Note that $M(\alpha_p) = \D_k/(F,V)$.

Similarly, if $M$ is a $\D_k$-module of finite length, the \emph{head}
(or \emph{co-socle}) of $M$ is its largest semisimple quotient.  We
write $\HH(M)$ for the head of $M$.  If $M$ is a $\D_k$-module on
which $F$ and $V$ are nilpotent, then $\HH(M)$ is a $k$-vector space
on which $F=V=0$.  The presentation of $M(w)$ by generators and
relations makes it clear that if
$w=v^{n_\ell}f^{m_\ell}\cdots v^{n_1}f^{m_1}$, then $\HH(M(w))$ has
dimension $\ell$.  More precisely, the images of the generators
$E_0,\dots,E_{\ell-1}$ in $\HH(M(w))$ are a basis.
 
\begin{defn}
  If $G$ is a finite group scheme over $k$ killed by $p$, the
  \emph{$a$-number of $G$}, denoted $a(G)$, is the largest integer $a$
  such that there is an injection $\alpha_p^a\into G$ of group
  schemes.  If $G$ is connected and unipotent, $p^{a(G)}$ equals the
  order of the foot of $G$.  Similarly, if $M$ is a $\D_k$-module of
  finite length, the \emph{$a$-number of $M$}, denoted $a(M)$, is the
  largest integer $a$ such that there is a surjection
  $M\onto M(\alpha_p)=\D_k/(F,V)$ of $\D_k$-modules.
  \end{defn}
  
  It is clear that the $a$-number is additive in direct sums and that
  $a(M(f))=a(M(v))=0$.

\begin{lemma} \label{Lanumber}
The $a$-numbers of $BT_1$ modules have the following
  properties: 
\begin{enumerate}
  \item If $\ell$, $m_1,\dots,m_\ell$, and $n_1,\dots,n_\ell$ are
    positive integers, then $a(M(v^{n_\ell}f^{m_\ell}\cdots
    v^{n_1}f^{m_1}))=\ell$. 
\item Also $a(M(w))$ is the number of rotations of
$w$ which start with $v$ and end with $f$.
\end{enumerate}
\end{lemma}

\begin{proof}
  For part (1),
  $a(M)=\dim_k(\HH(M))$ and, by the discussion above,
  if $w=v^{n_\ell}f^{m_\ell}\cdots v^{n_1}f^{m_1}$, then
  $a(M(w))=\ell$.  Part (2) is immediate from part (1).
\end{proof}

\subsection{Homs from $G_{1,1}$}\label{ss:homs-G11}

Write $M_{1,1}:=\D_k/(F-V) \cong M(fv)$.
Let $G_{1,1}$ be the $BT_1$ group scheme over $k$ 
such that $M(G_{1,1}) \simeq M_{1,1}$.  The group scheme
$G_{1,1}$ appears ``in nature'' as the kernel of multiplication by $p$
on a supersingular elliptic curve over $k$ (e.g., see
\cite[Prop.~4.1]{Ulmer91}).  

\begin{defns}\mbox{}
  Let $G$ be a $BT_1$ group scheme over $k$ and let $M$ be a $BT_1$
  module.
\begin{enumerate}
\item Define the \emph{$s_{1,1}$-multiplicity of $G$} as the largest
  integer $\mathfrak{s}$ such that there is an isomorphism of group
  schemes
  \[G\cong G_{1,1}^{\mathfrak{s}}\oplus G'.\]
Define the \emph{$s_{1,1}$-multiplicity of $M$} as the largest
  integer $\mathfrak{s}$ such that there is an isomorphism of
  $\D_k$-modules $M\cong M_{1,1}^{\mathfrak{s}}\oplus M'$.
\item 
Define the \emph{$u_{1,1}$-number of $G$} as the largest integer
$\mathfrak{u}$ such that there exists an injection
  \[ G_{1,1}^{\mathfrak{u}}\into G\]
  of group schemes. 
 Define the \emph{$u_{1,1}$-number of $M$}
  as the largest integer $\mathfrak{u}$ such that there is a
  surjection $M\onto M_{1,1}^{\mathfrak{u}}$ of $\D_k$-modules.  
  \end{enumerate}
\end{defns}

The notation $s_{1,1}$ (resp.\ $u_{1,1}$) is motivated by the word
superspecial (resp.\ unpolarized), see Section~ \ref{Sconnection}.
Note that $s_{1,1}(M(G))=s_{1,1}(G)$ and $u_{1,1}(M(G)) = u_{1,1}(G)$.
It is clear that $u_{1,1}(G) \geq s_{1,1}(G)$.  The
$s_{1,1}$-multiplicity and $u_{1,1}$-number are additive in direct
sums.

By the Kraft classification, if $M(G)$ is described
by a multiset of primitive cyclic words, then
$s_{1,1}(G)$ equals the multiplicity of the
cyclic word $fv$ in the multiset.

We want to compute the $u_{1,1}$-number of the standard $BT_1$ modules
$M(w)$.  Trivially,
\[u_{1,1}(M(f))=u_{1,1}(M(v))=0.\]
A straightforward exercise shows that $u_{1,1}(M(fv))=1$, and more
precisely that 
\[\Hom_{\D_k}(M(fv),M(fv))\cong\F_{p^2}\times k,\]
with $(c,d)\in\F_{p^2}\times k$ identified with the homomorphism
that sends the class of $1\in \D_k/(F-V)$ to the class of $c+dF$.  The
surjective homomorphisms are those where $c\neq0$. 

We may thus assume that $w$ has length $\lambda >2$.  We will evaluate
the $u_{1,1}$-number of $M(w)$ by computing
$\Hom_{\D_k}(M(w),M_{1,1})$ explicitly.  To that end, write
$w=v^{n_\ell}f^{m_\ell}\cdots v^{n_1}f^{m_1}$.  Since $\lambda > 2$
and and $w$ is non-periodic, we may replace $w$ with a rotated word so
that $m_1>1$ or $n_\ell>1$ (or both).  Roughly speaking, the following
proposition says that the $u_{1,1}$-number of $M(w)$ is the number of
appearances in $w$ of subwords of the form $v^{>1}(fv)^ef^{>1}$ where
$e\ge0$.  For example, if $w=v^2f^2v^3fvf^4$, then the
$u_{1,1}$-number is $2$.

\begin{propss}\label{prop:hom-G11-M(w)}
  Suppose $w=v^{n_\ell}f^{m_\ell}\cdots v^{n_1}f^{m_1}$ where
  $m_1>1$ or $n_\ell>1$.  Define $u$ by:
  \begin{multline*}
    u:=\#\left\{1\le i\le\ell \mid m_i>1
        \text{ and }n_i>1\right\}\\
    +
    \#\left\{1\le i<j\le\ell \mid n_j>1, \ m_j=n_{j-1}
        =\cdots=n_i=1,\text{ and }m_i>1\right\}.
  \end{multline*}
  Then
  \begin{enumerate}
  \item the $s_{1,1}$ number of $M(w)$ is $\ell$ if $m_i=n_i=1$ for
    all $i$, and 0 otherwise,
  \item $\Hom_{\D_k}(M(w),M_{1,1})$ is in bijection with
      $k^{u+\ell}$, and
    \item the $u_{1,1}$-number of $M(w)$ is $u$.
    \end{enumerate}
  \end{propss}

  \begin{proof}
    Part (1) follows from the fact that for a primitive word $w$, the
    $s_{1,1}$ number of $M(w)$ is 1 if $\overline w=\overline{fv}$ and
    is zero otherwise.
    
    For part (2), we use Lemma~\ref{Lrelations} to present $M(w)$ with
    generators $E_0,\dots,E_{\ell-1}$ (with indices taken modulo
    $\ell$) and relations $F^{m_i}E_{i-1}=V^{n_i}E_i$.  Let $z_0,z_1$
    be a $k$-basis of $M_{1,1}$ with $Fz_0=Vz_0=z_1$ and
    $Fz_1=Vz_1=0$.  Then a homomorphism $\psi:M(w)\to M_{1,1}$ is
    determined by its values on the generators $E_i$.  Write
\[\psi(E_i)=a_{i,0}z_0+a_{i,1}z_1.\]
Then $\psi$ is a $\D_k$-module homomorphism if and only if
$F^{m_i}\psi\left(E_{i-1}\right)=V^{n_i}\psi\left(E_i\right)$ for
$i=1,\dots,\ell$.  By \eqref{EbasicFV},
\[
F^{m_i}\psi\left(E_{i-1}\right) =  \begin{cases*}
    a_{i-1,0}^{p} z_1&\text{if $m_i=1$}\\
    0&\text{if $m_i>1$}
  \end{cases*} \text{ and }
  V^{n_i}\psi\left(E_i\right) =
  \begin{cases}
    a_{i,0}^{1/p} z_1&\text{if $n_i=1$}\\
    0&\text{if $n_i>1$}.
  \end{cases}
\]
This system of equations places no constraints on $a_{i,1}$ for
$i = 1, \ldots, \ell$, because $Vz_1=Fz_1=0$.  The constraints on
$a_{i,0}$ for $i=1, \ldots, \ell$ are: if $m_i=n_i=1$, then
$a_{i-1,0}^p = a_{i, 0}^{1/p}$; if $m_i = 1$ and $n_i >1$, then
$a_{i-1,0}=0$; if $m_i>1$ and $n_i=1$, then $a_{i,0} = 0$; if
$m_i > 1$ and $n_i >1$, then no constraint.  Using that $m_1>1$ or
$n_\ell>1$, we find a triangular system of equations for $a_{i,0}$,
and it is a straightforward exercise to show that the solutions are in
bijection with $k^u$.  Combining with the $k^\ell$ unconstrained
values of $\{a_{i,1} \mid 1 \leq i \leq \ell\}$ yields part (2).

For part (3), the homomorphism $\psi:M(w)\to M_{1,1}$ is surjective if
and only if at least one of the $a_{i,0}$ is not zero, which is
equivalent to $\psi$ inducing a surjection
$\HH(M(w))\to\HH(M_{1,1})=k$.  Part (1) implies that there are $u$
independent such $\psi$ (and no more).  This shows that $u$ is the
largest integer such that there is a surjection $M(w)\to M_{1,1}^u$,
completing the proof of part (3).
\end{proof}

\subsection{Motivation}
The next proposition motivates 
the
$s_{1,1}$-multiplicity and the $u_{1,1}$-number.

\begin{propss}
Let $A/k$ be an abelian variety and let $E/k$ be a
  supersingular elliptic curve.
  \begin{enumerate}
    \item If there is an abelian variety $B$ and an isogeny $E^s\times
      B\to A$ of degree prime to $p$, then $s$ is less than or equal
      to the $s_{1,1}$-multiplicity of $A[p]$.
    \item If there is a morphism of abelian varieties $E^u\to A$ with
      finite kernel of order prime to $p$, then $u$ is less than or
      equal to the $u_{1,1}$-number of $A[p]$.
\end{enumerate}
\end{propss}

\begin{proof}
  An isogeny as in (1) shows that $E[p]\cong G_{1,1}$ is a direct factor of
  $A[p]$ with multiplicity $s$, so $s \leq s_{1,1}(A[p])$.
  A morphism as in (2) shows that $E[p]\cong G_{1,1}$ appears in
  $A[p]$ with multiplicity at least $u$, so $u \leq u_{1,1}(A[p])$.
\end{proof}

\subsection{Connection to the superspecial rank} \label{Sconnection}
Suppose $G$ is a polarized $BT_1$ group scheme.
In \cite[Def.~3.3]{AchterPries15}, Achter and Pries define the
\emph{superspecial rank} of $G$ as the largest integer $s$ such that
there is an injection $G_{1,1}^s\into G$ and such that the
polarization on $G$ restricts to a non-degenerate pairing on
$G_{1,1}^s$.  In this situation, the pairing allows them to define a
complement (see \cite[Lemma~3.4]{AchterPries15}), so that
$G\cong G_{1,1}^s\oplus G'$ (a direct sum of polarized $BT_1$ group
schemes).  A self-dual $BT_1$ group scheme $G$ equipped with a
decomposition $G\cong G_{1,1}\oplus G'$ (just of $BT_1$ group schemes)
automatically admits a polarization compatible with the direct sum
decomposition.  Therefore, the superspecial rank of $G$ equals
its $s_{1,1}$-multiplicity. 

They also define an \emph{unpolarized superspecial rank}, which is the
same as our $u_{1,1}$-number, and prove a result
\cite[Lemma~3.8]{AchterPries15} which is closely related to and
implied by Proposition~\ref{prop:hom-G11-M(w)}.

Next, we consider a correction to \cite[Thm.~3.14]{AchterPries15}.
Let $K$ be the function field of an irreducible, smooth, proper curve
$X$ over $k$, let $J_X$ be the Jacobian of $X$, and let $E$ be a
supersingular elliptic curve over $k$ which we regard as a curve over
$K$ by base change.  Let $\Sel(K,p)$ denote the Selmer group for the
multiplication-by-$p$ isogeny of $E/K$.  See \cite{Ulmer91} for details.

\begin{prop}
With notation as above, 
let $a$, $u_{1,1}$, and $s_{1,1}$ be the $a$-number, $u_{1,1}$-number,
and $s_{1,1}$-multiplicity of $J_X[p]$ respectively. 
Then $\Sel(K,p)$ is isomorphic to the product of
  a finite group and a $k$-vector space of dimension
  $a+u_{1,1}-s_{1,1}$.
\end{prop}

\begin{proof}
  Applying \cite[Props.~6.2 and 6.4]{Ulmer19} with $\CC=X$, $\DD=E$,
  $\Delta=1$, and $n=1$ shows that $\Sel(K,p)$ differs by a finite
  group from $\Hom_{\D}(H^1_{dR}(X),M_{1,1})$.  By
  \cite[Cor.~5.11]{Oda69}, $H^1_{dR}(X)\cong M(J_X[p])$.  Write
  $M(J_X[p])$ as a sum of indecomposable $BT_1$ modules $M(w)$ for
  suitable cyclic words $w$.  In Section~4.3, we computed
  $H_w=\Hom_{\D}(M(w),M_{1,1})$ for a primitive cyclic word $w$.
  Recall that $H_w=0$ if $w = f$ or $w=v$ and
  $H_w \cong \F_{p^2}\times k$ if $w = fv$.  Also
  $H_w \cong k^{u+\ell}$, if $w\not\in\{1, f,v,fv\}$ where
  $u=u_{1,1}(M(w))$ and $\ell$ is the $a$-number of $M(w)$.  The
  result follows from the additivity of $a$-numbers,
  $u_{1,1}$-numbers, and $s_{1,1}$-multiplicities.
 \end{proof}

 Next, we compute the numerical invariants of a module
 $M=\oplus M(\overline w_i)^{m_i}$ in terms of multiplicities of words.
 In Section~\ref{ss:words-to-canonical}, we defined
 multiplicities $\mu$ by considering all lifts of $\overline{w}_i$ to
 words and taking powers of those words so they all have the same
 length $\ell=\lcm(\text{lengths of }\overline w_i)$. Let
 $\mu(f\text{---}v)$ be the sum of the multiplicities of all these words
of length $\ell$ starting with $f$ and ending with $v$; and let
$\mu(\text{---}vf)$ be the sum of the multiplicities of all words of
length $\ell$ ending with $vf$.

\begin{prop}\label{prop:FEO-general-numbers}
The $BT_1$ module $M=\oplus M(\overline w_i)^{m_i}$ has
  \begin{enumerate}
  \item $p$-rank equal to $\mu(f^\ell)$,
  \item $a$-number equal to
    $\mu(f\text{---}v)=\mu(\text{---}vf)=\mu(\text{---}fv)$,
  \item $s_{1,1}$-multiplicity equal to $\mu((fv)^{\ell/2})$
    if $\ell$ is even and to $0$ if $\ell$ is odd,
  \item and $u_{1,1}$-number equal to the sum of the $s_{1,1}$
    multiplicity and
    \[ \mu(\text{---}v^2f^2)+\mu(\text{---}v^2fvf^2)+\cdots+
      \mu(\text{---}v^2(fv)^{\lfloor (\ell-4)/2\rfloor}f^2)
   .\]
  \end{enumerate}
\end{prop}

\begin{proof}
  Part (1) (resp.\ part (2)) follows from Lemma~\ref{Lprank} (resp.\
  Lemma~\ref{Lanumber}).  Parts (3) and (4) follow from
  Proposition~\ref{prop:hom-G11-M(w)}.
\end{proof}

\subsection{Examples}
For small genus, we give tables of elementary sequences (``E--O''),
matched with the self-dual multisets of primitive cyclic words
(``K''), together with their $p$-ranks, $a$-numbers,
$s_{1,1}$-multiplicities, and $u_{1,1}$-numbers.

From Section~\ref{Selemseq}, for the $BT_1$ module of the elementary
sequence $\Psi= [\psi_1,\dots,\psi_g]$, the $p$-rank is the largest
$i$ such that $\psi_i=i$ and the $a$-number is $g-\psi_g$.  We do not
know how to compute the $s_{1,1}$-multiplicity or $u_{1,1}$-number
directly from $\Psi$.

For the $BT_1$ module of a multiset of cyclic words, the $p$-rank is
the multiplicity of the word $f$ by Lemma~\ref{Lprank} and the
$a$-number can be computed using Lemma~\ref{Lanumber}.  The
$s_{1,1}$-multiplicity is the multiplicity of the cyclic word $fv$ and
the $u_{1,1}$-number can be computed using
Proposition~\ref{prop:hom-G11-M(w)}.

\begin{center}
\renewcommand{\arraystretch}{1.2}
\begin{tabular}{| c | c | c | c | c | c |}
\hline
\multicolumn{6}{|c|}{$g=1$} \\
\hline
E--O&K&$p$-rank&$a$-number&$s_{1,1}$-mult.&$u_{1,1}$-number\\
\hline
$[0]$&$\{fv\}$&0&1&1&1\\
\hline
$[1]$&$\{f,v\}$&1&0&0&0\\
\hline
\end{tabular}
\end{center}

\begin{center}
\renewcommand{\arraystretch}{1.2}
\begin{tabular}{| c | c | c | c | c | c |}
\hline
\multicolumn{6}{|c|}{$g=2$} \\
\hline
 E--O&K&$p$-rank&$a$-number&$s_{1,1}$-mult.&$u_{1,1}$-number\\
\hline
$[0,0]$&$\{(fv)^2\}$&0&2&2&2\\
\hline
$[0,1]$&$\{ffvv\}$&0&1&0&1\\
\hline
$[1,1]$&$\{f,v,fv\}$&1&1&1&1\\
\hline
$[1,2]$&$\{(f)^2,(v)^2\}$&2&0&0&0\\
\hline
\end{tabular}
\end{center}


If $G$ is a polarized $BT_1$ group scheme of order $p^{2g}$ with
positive $p$-rank, then $G\cong G'\oplus\Z/p\Z\oplus\mu_p$ where $G'$
is a polarized $BT_1$ group scheme of order $p^{2g-2}$.  Thus the rows
of the table for $g$ for $G$ with positive $p$-rank can be deduced
from the table for $g-1$.  In passing from genus $g-1$ to genus $g$:
the elementary sequence changes from $[\psi_1,\dots,\psi_{g-1}]$ to
$[1,\psi_1+1,\dots,\psi_{g-1}+1]$; the multiplicity of the words $f$
and $v$ increases by 1; the $p$-rank increases by 1; and the
$a$-number, $s_{1,1}$-multiplicity, and $u_{1,1}$-number stay the
same.  In light of this, when $g=3$ and $g=4$, we only include the
$BT_1$ group schemes with $p$-rank $0$ in the table.

\begin{center}
\renewcommand{\arraystretch}{1.2}
\begin{tabular}{| c | c | c | c | c | c |}
\hline
\multicolumn{6}{|c|}{$g=3$} \\
\hline
 E--O&K&$p$-rank&$a$-number&$s_{1,1}$-mult.&$u_{1,1}$-number\\
\hline
 $[0,0,0]$&$\{(fv)^3\}$&0&3&3&3\\
  \hline
 $[0,0,1]$&$\{fv,ffvv\}$&0&2&1&2\\
  \hline
 $[0,1,1]$&$\{fvv,vff\}$&0&2&0&0\\
  \hline
 $[0,1,2]$&$\{fffvvv\}$&0&1&0&1\\
  \hline
\end{tabular}
\end{center}

The E--O structure $[0,1,1]$ is the first which is decomposable as a
$BT_1$ group scheme but indecomposable as a polarized $BT_1$ group
scheme.

\begin{center}
\renewcommand{\arraystretch}{1.2}
\begin{tabular}{| c | c | c | c | c | c |}
\hline
\multicolumn{6}{|c|}{$g=4$} \\
\hline
 E--O&K&$p$-rank&$a$-number&$s_{1,1}$-mult.&$u_{1,1}$-number\\
\hline
 $[0,0,0,0]$&$\{(fv)^4\}$&0&4&4&4\\
  \hline
 $[0,0,0,1]$&$\{(fv)^2,ffvv\}$&0&3&2&3\\
  \hline
 $[0,0,1,1]$&$\{ffvfvvfv\}$&0&3&0&1\\
  \hline
 $[0,0,1,2]$&$\{(ffvv)^2\}$&0&2&0&2\\
  \hline
 $[0,1,1,1]$&$\{fv,ffv,vvf\}$&0&3&1&1\\
  \hline
 $[0,1,1,2]$&$\{fv,fffvvv\}$&0&2&1&2\\
  \hline
 $[0,1,2,2]$&$\{fffv,fvvv\}$&0&2&0&0\\
  \hline
 $[0,1,2,3]$&$\{ffffvvvv\}$&0&1&0&1\\
  \hline
\end{tabular}
\end{center}

\section{Fermat Jacobians}\label{s:FJ}
In this section, we recall two results on the $BT_1$ modules
of Fermat curves from \cite{PriesUlmerBT1s}.

For each positive integer $d$ not divisible by $p$, let $F_d$ be the
Fermat curve of degree $d$, i.e., the smooth, projective curve over $k$
with affine model
\[F_d:\qquad X^d+Y^d=1,\]
and let $J_{F_d}$ be its Jacobian.
Let $\CC_d$ be the smooth, projective curve over $k$ with affine model 
\begin{equation}\label{eq:Cd-model}
\CC_d:\qquad y^d=x(1-x).  
\end{equation}
Then $\CC_d$ is a quotient of $F_d$.  (Substitute $X^d$ for $x$ and
$XY$ for $y$ in the equation for $\CC_d$.)  The map $F_d \to \CC_d$ is
the quotient of $F_d$ by the subgroup
\[\{(\zeta,\zeta^{-1})\mid\zeta\in\mu_d\} \subset(\mu_d)^2 \subset
  {\rm Aut}(F_d)\]
of index $d$.  The Riemann-Hurwitz formula shows that the genus of
$\CC_d$ is
\[g(\CC_d)=\lfloor (d-1)/2\rfloor=
  \begin{cases}
  (d-1)/2&\text{if $d$ is odd,}\\
    (d-2)/2&\text{if $d$ is even.}
  \end{cases}\]

\begin{thm}[={\cite[Thm.~5.5]{PriesUlmerBT1s}}]\label{thm:BT1Cd}
  The Dieudonn\'e module $M(J_d[p])$ is the $BT_1$ module with
  data
  \begin{equation*}
    S= \Z/d\Z\setminus\{0,d/2\} \text{ if $d$ is even, and }
    S=\Z/d\Z\setminus\{0\} \text{ if $d$ is odd},
\end{equation*}
\begin{equation*}
  S_f =\left\{a\in S \mid d/2<a<d \right\}, \quad
  S_v=\left\{a\in S \mid 0<a<d/2 \right\},
\end{equation*}
and the permutation $\pi:S\to S$ given by $\pi(i)=pa$.
\end{thm}

\begin{thm}[={\cite[Thm.~5.9]{PriesUlmerBT1s}}]\label{thm:BT1Fd}
  The Dieudonn\'e module $M(J_{F_d}[p])$ is the $BT_1$ module with
  data
   \begin{equation*}
  T =\left\{(a,b)\in(\Z/d\Z)^2 \mid a\neq0,b\neq0,a+b\neq0 \right\},
  \end{equation*}
  \begin{equation*}
    T_f =\left\{(a,b)\in S \mid a+b>d\right\},
    \quad T_v=\left\{(a,b)\in S \mid a+b<d\right\},
  \end{equation*}
  and the permutation $\sigma(a,b)=(pa,pb)$.
\end{thm}

\begin{rem}
  The quotient $\CC_d\to F_d$ induces an inclusion $S\into T$ sending
  $a$ to $(a,a)$, and the partition and permutation of $T$ are
  compatible with those of $S$.  There are other quotients $\CC$ of
  $F_d$ by subgroups of $\mu_d^2$.  (For example, the curves
  $v^e=u^r(1-u)^s$ where $e \mid d$, $\gcd(r,s,e)=1$, and $r+s<e$.)
  Each gives rise to a set $S_\CC$ with partition and permutation and
  the projection $F_d\to\CC$ induces an inclusion $S_\CC\into T$
  compatible with the partitions and permutations.  The interesting
  features of $J_{F_d}[p]$ are already present in $J_{\CC_d}[p]$, so
  we restrict to studying this case for simplicity.
\end{rem}

\section{Ekedahl--Oort structures of Fermat quotients:
  general case}\label{s:FEO-general}
In this section, we determine the Ekedahl--Oort type of the Jacobian
$J_d$ of the Fermat quotient curve $\CC_d$ with affine model
$y^d=x(1-x)$ for $d$ relatively prime to $p$.

\subsection{Words, patterns, and multiplicities}
By Theorem~\ref{thm:BT1Cd}, the $BT_1$ module of $J_d[p]$ is the one
obtained from diagram~(3.1) using the data $S\subset \Z/d\Z$ with its
usual partition and permutation.  This data gives rise to a multiset
of cyclic words.  To compute the E--O structure, as in 
Section~\ref{ss:words-to-canonical}, we consider all representative
words of the cyclic words, take powers so they all have the same
length, and compute their multiplicities to find the dimensions of the
blocks $B_i=M_{i+1}/M_i$ in the canonical filtration.

We want to reformulate this to go directly from the set $S$ to the
multiplicities $\mu$.  To do so, we introduce the \emph{pattern} of an
element of $S$.  This is a variant of the word which takes into
account the powers mentioned in the previous paragraph.

Let $\ell=|\<p\>|$ be the multiplicative order of $p$ modulo $d$.
Let $\WW_\ell$ be the set of words of length $\ell$ on
$\{f,v\}$.  Define a map $\pat:S\to\WW_\ell$ as follows: for $a\in S$,
define $\pat(a)=u_{\ell-1}\cdots u_0$ where
\[u_j = f \text{ if } p^ja \in S_f, \text{ and } u_j = v \text{ if } p^ja \in S_v.\]
If the word $w_a$ for $a$ has length $\ell$ (which happens
when $\gcd(a,d) = 1$), then $\pat(a)=w_a$
whereas for $a$ with a shorter word, $\pat(a)$ is a
power of $w_a$.

For example, take $d=9$ and $p=2$, so that $\ell=6$.  The orbits of
$\<p\>$ are $1\to2\to4\to8\to7\to5\to1$ and $3\to6\to3$.  For $a=3$,
the word is $fv$ and $\pat(3)=fvfvfv$.  For $a=6$, the word is $vf$
and $\pat(6)=vfvfvf$.  For $a\not = 3,6$, then $\pat(a)=w_a$; for
example $\pat(1)=fffvvv$.

For $w\in\WW_\ell$, define the \emph{multiplicity} of $w$ to be the
cardinality of its inverse image under $\pat$:
\[\mu(w):=\left|\pat^{-1}(w)\right|.\]
This is the same multiplicity as defined in
Section~\ref{ss:words-to-canonical} for the module determined by $S$.

\subsection{Ekedahl--Oort structure of $J_d[p]$}
Recall from Section~\ref{Selemseq} that the E--O structure associated
to a self-dual $BT_1$ module is the sequence of integers
$[\psi_1,\dots,\psi_g]$ which starts from $\psi_0=0$ and has sections
of length equal to the block sizes $\dim(B_i)$ which are increasing
(resp.~constant) if the word attached to the block ends with $f$
(resp.~$v$).

To simplify this description, 
we introduce some notation: $\nearrow^m$ (resp.\ $\to^m$) stands for
an increasing (resp.\ constant) sequence of integers of length $m$.
Thus,
\[ [\nearrow^3\to^2]=[1,2,3,3,3]\quad\text{and}\quad
  [\to^2\nearrow^3]=[0,0,1,2,3].\]

Now enumerate the elements of $\WW_\ell$ that start with $f$ in
lexicographic order:
\begin{equation} \label{Ewords}
w_0=f^\ell, \ w_1=f^{\ell-1}v, \ w_2=f^{\ell-2}vf, \ w_3=f^{\ell-2}vv, \ 
\dots, \ w_{2^{\ell-1}-1}=fv^{\ell-1}.
\end{equation}
Let $\mu_j=\mu(w_j)$.

Using Sections~\ref{ss:words-to-canonical} and \ref{Selemseq}, we get
the following description of the E--O structure of $J_d[p]$:

\begin{thm}\label{thm:FEO-general}
  Let $\ell$ be the multiplicative order of $p$ modulo $d$, and let
  $\mu_0, \ldots, \mu_{2^{\ell-1}-1}$ be the multiplicities of the
  words above.  Let $J_d$ be the Jacobian of the Fermat quotient curve
  $C_d$ with affine model $y^d=x(1-x)$.  The Ekedahl--Oort structure
  of $J_d[p]$ is given by
  \[ [\nearrow^{\mu_0}\to^{\mu_1}\nearrow^{\mu_2}
    \cdots\to^{\mu_{2^{\ell-1}-1}}].\]
\end{thm}

\begin{proof}
  Indeed, for the $BT_1$ module $M=M(J_d[p])$ and for
  $0 \leq j < 2^{\ell-1}$, the subspace $M_{j+1}$ in the canonical
  filtration is the span of the basis vectors $e_a$ indexed by
  $a\in S$ such that $\pat(a)\le w_j$.  In particular,
  ${\rm dim}(M_{j+1})=\sum_{i=0}^{j} \mu_i$.  By
  Definition~\ref{Dcoarse}, $B_{j}=M_{j+1}/M_{j}$, so
  ${\rm dim}(B_{j})=\mu_j$.  If $j$ is even, then $w_j$ ends with $f$
  so $B_{j}$ is mapped isomorphically onto its image by $F$.  If $j$
  is odd, then $w_j$ ends with $v$ so $B_{j}$ is killed by $F$.  Thus
  the $\mu_j$ are the values of $\rho(j+1)-\rho(j)$ in the canonical
  type of $J_d[p]$, and they give the lengths of the runs where the
  elementary sequence is increasing ($j$ even) or constant ($j$ odd).
\end{proof}

\section{Ekedahl--Oort structures of Fermat quotients:
  the case $p=2$}\label{s:FEO-p=2}

In this section, suppose $p=2$.  In this case, the curve $\CC_d$ is an
Artin-Schreier cover of the projective line with equation $x^2-x=y^d$,
and the formulas when $p=2$ are 
different from the case when $p$ is odd.  We do not include proofs in
this section because most of the results already appear in the
literature.

The genus of $\CC_d$ is $g_d=(d-1)/2$.

\begin{cor} \label{Cp=2} Let $p=2$ and $d>1$ be odd.  Let $J_d$ be the
  Jacobian of $\CC_d:y^d=x(1-x)$.  Then:
\begin{enumerate}
\item \cite[special case of Thm.~1.3]{ElkinPries13} the Ekedahl--Oort
  type of $J_d[2]$ has the form
  \[ [0,1,1,2,2,\dots, \lfloor g_d/2 \rfloor].\]
 \item \cite[Theorem 4.2]{Subrao75} (Deuring-Shafarevich formula) the
   $2$-rank of $J_d[2]$ is $0$;
\item \cite[Prop.~3.4]{ElkinPries13} the $a$-number of $J_d[2]$ is
  $\frac{d-1}4$ if $d\equiv 1\pmod 4$ and $\frac{d+1}4$ if
  $d\equiv 3\pmod 4$;
\item \cite[Application 5.3]{AchterPries15} the $s_{1,1}$-multiplicity
  of $J_d[2]$ is $1$ if $d \equiv 0 \bmod 3$ and is $0$ otherwise.
\end{enumerate}
\end{cor}


\section{The $a$-number of the Fermat quotient curve}\label{s:anumber}

Suppose $p$ is odd.  Let $J_d$ be the Jacobian of the curve $\CC_d$
with affine model $y^d=x(1-x)$.  We find a closed-form formula for the
$a$-number of $J_d$ and some information about its $p$-rank.

If $d=1,2$, then $\CC_d$ is rational; we exclude this trivial case in
the next results.

\begin{prop} \label{Pforma} Suppose $p$ is odd, $d>2$, and $p$ does
  not divide $d$.  Then:
  \begin{enumerate}
  \item The $a$-number of $J_d[p]$ is
    \[\sum_{j=1}^{(p-1)/2}
      \left(\left\lfloor\frac{2jd}{2p}\right\rfloor
        -\left\lfloor\frac{(2j-1)d}{2p}\right\rfloor \right)
      =  \frac{(p-1)}p\frac d4- \sum_{j=1}^{(p-1)/2}
      \left(  \left\<\frac{2jd}{2p}\right\>
     -\left\<\frac{(2j-1)d}{2p}\right\> \right).\]
Here $\<\cdot\>$ denotes the fractional part.
\item If $d\equiv\pm1\pmod{2p}$, then
  the $a$-number of $J_d$ is $(p-1)(d \mp 1)/4p$.
  \item If $d\equiv p\pm1\pmod{2p}$, then the
    $a$-number of $J_d$ is $(p-1)(d \pm (p-1))/4p$.
  \end{enumerate}
\end{prop}

\begin{rem}
An analogue of parts (2) and (3) for 
the Fermat curve $F_d$ was proven by Montanucci
and Speziali \cite{MontanucciSpeziali18} using
the Cartier operator. 
\end{rem}

\begin{proof}
By Proposition~\ref{prop:FEO-general-numbers}, the $a$-number of
$J_d[p]$ is $\mu(\text{---}fv)$, the number of elements $a\in S$ such
that the pattern of $a$ ends with $fv$.  These are precisely the
elements with $a\in S_v$ and $pa\in S_f$, and we may count them using
archimedean considerations.
More precisely, $a\in S_v$ means $0<a<d/2$ and $pa\in S_f$ means that
(the least positive residue of) $pa$ satisfies $d/2<pa<d$.  

We prove part (1) and leave the other parts as exercises.
The elements of
$S$ which contribute to the $a$-number are represented by integers
satisfying one of the inequalities
\[\frac{d}{2p}<a<\frac{2d}{2p},\quad
  \frac{3d}{2p}<a<\frac{4d}{2p},\quad\dots,\quad
  \frac{(p-2)d}{2p}<a<\frac{(p-1)d}{2p},\]
and the number of such integers is the left hand side of the displayed
equation in part (1).  The equality in the displayed equation in part
(1) is immediate from the definitions of $\lfloor.\rfloor$ and
$\<.\>$.
\end{proof}

\begin{rem}
If $p$ is large and $d$ is large with respect to $p$,
by Proposition~\ref{Pforma}, the $a$-number of $J_d$ is close to
$\frac{(p-1)}p\frac d4$ which is close to $g/2$.  The
second expression in part (1) shows that the difference is less than $(p-1)/2$
in absolute value.  Numerical experiments suggest that
the difference is bounded by $\frac{(p-1)^2}{4p}$
and part (3) shows that the difference equals this when $d\equiv
p\pm1\pmod{2p}$. 
\end{rem}

An abelian variety $A$ is \emph{superspecial} if its $a$-number is
equal to its dimension.  This is equivalent to $A$ being isomorphic to
a product of supersingular elliptic curves. The next result shows that
$J_d$ is superspecial if and only if $\CC_d$ is a quotient of
$\CC_{p+1}$ by a subgroup of $\mu_{p+1}$.

\begin{prop}\label{prop:superspecial}
 Suppose $d>2$ and $p \nmid d$.  Then $J_d$
  is superspecial if and only if $d$ divides $p+1$.  
\end{prop}

An analogue of Proposition~\ref{prop:superspecial} for $F_d$ was
proven by Kodama and Washio \cite[Cor.~1, p.~192]{KodamaWashio88}.
The ``if'' direction of our proposition
follows from their result.

\begin{proof}
  By Proposition~\ref{prop:FEO-general-numbers}(2), $J_d$ is superspecial
  if and only if $\mu((fv)^{g/2})=g$.  This is the case if and only if
  $p$ exchanges $S_f$ and $S_v$, i.e., $pS_f=S_v$ and $pS_v=S_f$.
  Note also that if this statement holds for $(p,d)$ then it holds for
  $(p,d')$ for any divisor $d'>2$ of $d$.

  We claim that if $H\subset(\Z/d\Z)^\times$ is a subgroup such that
  $hS_v=S_v$ (equivalently $hS_f=S_f$) for all $h\in H$, then
  $H=\{1\}$.  Suppose that $H$ is such a subgroup, $1<a<d$, and the
  class of $a$ lies in $H$.  If $a>d/2$, then $a$ sends $1\in S_v$ to
  $a\not\in S_v$, a contradiction.  If $a<d/2$, then there is an
  integer $b$ in the interval $(d/2a,d/a)$, so $b\in S_v$, and
  $d/2<ab<d$, so the class of $ab$ lies in $S_f$, again a
  contradiction.  Thus $H=\{1\}$.

  Suppose that $p$ exchanges $S_f$ and $S_v$.  By applying the claim
  to $H=\<p^2\>$, we see that $p$ has order $2$ modulo $d$, and the
  same holds for $p$ modulo $d'$ for any divisor $d'$ of $d$.  Since
  $1$ does not exchange $S_f$ and $S_v$, the order of $p$ is exactly
  $2$ modulo any $d'>2$ dividing $d$.  If $d'$ is an odd prime power,
  this implies $p\equiv-1\pmod{d'}$.
  
  More generally, let $d'=2^e$ be the largest power of $2$ dividing
  $d$.  If $d'=1$ or $2$, then $p\equiv-1\pmod{d'}$.  If $d'=4$, there
  is a unique class of order exactly 2 modulo $d'$, namely $-1$, and
  again $p\equiv-1\pmod {d'}$.  Finally, if $e>2$, then there are
  three elements of order exactly 2 modulo $d'$, but only one of them
  reduces to an element of order 2 modulo $2^{e-1}$, namely $-1$.
  Again, we find $p\equiv-1\pmod{d'}$.  In all three cases,
  $p\equiv-1\pmod{d'}$ and so $p\equiv-1\pmod d$, as required.
\end{proof}

\subsection{Observations about the $p$-rank}

The behavior of the $p$-rank, $s_{1,1}$-multiplicity, and
$u_{1,1}$-number of $J_d$ for arbitrary $p$ and $d$ seems rather
erratic.  In later sections, we give closed form formulas for these
invariants under restrictions on $d$.  Here we include some
observations about when the $p$-rank is as large or small as possible.

An abelian variety is \emph{ordinary} if its
$p$-rank is equal to its dimension.  
The next result shows that $J_d$ is ordinary if and only if $\CC_d$
is a quotient of $\CC_{p-1}$ by a subgroup of $\mu_{p-1}$.

\begin{prop}\label{prop:ordinary}
 Suppose $d>2$ and $p \nmid d$.  
 Then $J_d$ is ordinary if and only if $d$ divides $p-1$. 
\end{prop}

An analogue of Proposition~\ref{prop:ordinary} for the Fermat curve
$F_d$ was proven by Yui \cite[Thm.~4.2]{Yui80} using exponential sums;
(see also \cite[Prop~5.1]{Gonzalez97} for a proof using the Cartier
operator when $d$ is prime).  The ``if'' direction of our proposition
can be deduced from Yui's result.

\begin{proof}
By Proposition~\ref{prop:FEO-general-numbers}(1), 
  $J_d$ is ordinary if and only if
  $\mu(f^\ell)=g$. 
  If $d$ divides $p-1$, then
  $\ell=1$, the orbits of $\<p\>$ on $S\subset\Z/d\Z$ are singletons,
  and $\mu(f)=|S_f|=g$, so $J_d$ is ordinary.

  The converse follows from the claim in the proof of
  Proposition~\ref{prop:superspecial}. \end{proof}

At the opposite extreme, the $p$-rank of $J_d$
is $0$ when $d$ divides $p+1$, by Proposition~\ref{prop:superspecial}.
If $J_d$ is supersingular, then the $p$-rank of $J_d$ is zero (the converse is
not necessarily true for $g \geq 3$); by \cite[Prop.~5.1]{Gonzalez97},
$J_d$ is supersingular when $d$ is prime and the order of $p$ in
$(\Z/d\Z)^\times$ is even.  

\subsection{Breaks in words}

We end this section with an elementary combinatorial result used
later. 

Fix $d$ and let $\ell$ be the order of $\<p\>\subset(\Z/d\Z)^\times$.
Given a pattern of length $\ell$, say $w=u_{\ell-1}\cdots u_0$, we say
that $0\le j<\ell-1$ is a \emph{break} of $w$ if $u_{j+1}\neq u_j$,
and we say $j=\ell-1$ is a break of $w$ if $u_0\neq u_{\ell-1}$.  If
$k$ is the number of breaks of $w$, then $k$ is even and
$0\le k\le\ell$.  Moreover, a pattern $w$ is determined by its set of
breaks and by its last letter $u_0$, and there are $2\binom \ell k$
words of length $\ell$ with $k$ breaks.  The sum of
these numbers for $0 \leq k \leq \ell$ equals $2^\ell$.

We also consider ``self-dual'' words of length $\ell=2\lambda$, i.e.,
words of the form $w^c\cdot w$ where $w$ has length $\lambda$.  Such a
word is determined by its last half $w$, and we may encode $w$ by
specifying its last letter and its ``breaks'' as above, ignoring the
last potential break: If $w=u_{\lambda-1}\cdots u_0$ we say that
$0\le j<\lambda-1$ is a \emph{break} if $u_{j+1}\neq u_j$.  (We ignore
$j=\lambda-1$ because whether or not $u_{\lambda-1}$ is a break of
$w^c \cdot w$ is already determined by the other data.)  There are
$2\binom{\lambda-1}k$ words $w$ with $k$ breaks.
The sum of these numbers for $0 \leq k \leq \lambda-1$ is
$2^\lambda$.

Fix $\ell \geq 1$.  As in \eqref{Ewords}, list all words of length
$\ell$ which begin with $f$ in lexicographic order:
\[w_0=f^\ell, w_1=f^{\ell-1}v,
  w_2=f^{\ell-2}vf,\dots,w_{2^{\ell-1}-1}=fv^{\ell-1}.\] Let $k(i)$ be
the number of breaks of $w_i$ (not looking at possible wrap-around
breaks).

\begin{lemma}\label{lemma:k} 
  The function $k(i)$ has the properties: \textup{(i)}
  $k(0)=0$; and \textup{(ii)} if $2^j\le i<2^{j+1}$, then
  $k(i)=k(2^{j-1}-1-i)+1$, and it is characterized by these
  properties.  
\end{lemma}

\begin{proof}
  Clearly $k(0)=0$.  In the list of words, $w_i$ is the binary
  representation of the integer $i$ where $f$ stands for 0, $v$ stands
  for $1$, and the leftmost letters are the most significant digits.
  Thus, if $2^j\le i<2^{j+1}$, then $w_i=f^{\ell-j-1}vt$ and
  $w_{2^{j-1}-1-i}=f^{\ell-j}(t^c)$ for some $t$, and it is visible
  that $w_i$ has one more break than $w_{2^{j-1}-1-i}$ does.  This
  proves the second property of $k$.  The two properties clearly
  characterize $k$.
\end{proof}

The function $i\mapsto k(i)$ is independent of $\ell$ if
$2^{\ell-1}>i$.    Its first few values are:

\begin{center}
\begin{tabular}{| c | c | c | c | c | c | c | c |  c | c | c | c | c | c | c | c | c|}
\hline
$i$&0&1&2&3&4&5&6&7&8&9&10&11&12&13&14&15\\
\hline
$k(i)$&0&1&2&1&2&3&2&1&2&3&4&3&2&3&2&1\\
  \hline
\end{tabular}
\end{center}

\section{$J_d[p]$ in the ``encompassing'' case}\label{s:encompassing}

Suppose $p$ is odd.  In this section, we fix $d=p^\ell-1$, which we
call the ``encompassing'' case.  The reason is that if $p \nmid d'$,
then $d'$ divides $p^\ell-1$ for some $\ell$, with the quotient
$(p^\ell-1)/d'$ being prime-to-$p$, and so $J_{d'}[p]$ is a direct
factor of $J_{p^\ell-1}[p]$.  Let $S = S_f \cup S_v$ and $\pi$ be
defined as before.

\subsection{$p$-adic digits}
Elements $a\in S$ correspond to $p$-adic expansions
\begin{equation} \label{Epadic}
a=a_0+a_1p+\cdots+a_{\ell-1}p^{\ell-1}.
\end{equation}
where $a_i\in\{0,\dots,p-1\}$ and we exclude the following cases: all
$a_i=0$ (when $a=0$); all $a_i=p-1$ (when $a=d$); and
the case all $a_i=(p-1)/2$ (when $a=d/2$).  Multiplication by $p$
corresponds to permuting the digits cyclically.

\subsection{Multiplicities}
Given $a\in S$, let $\pat(a)=u_{\ell-1}\cdots u_0$ be the pattern of
$a$.
Let \eqref{Epadic}
be the $p$-adic expansion of $a$.
Then $a \in S_v$ (meaning $u_0=v$) if and only if $a<d/2$.  
In other words, the condition is that the first $p$-adic digit to the left of
$a_{\ell-1}$ (inclusive) which is not $(p-1)/2$ is in fact
less than $(p-1)/2$.  

Similarly $a\in S_f$ (meaning $u_0=f$) if and only if $a>d/2$.  This
is true if and only if the first $p$-adic digit to the left of
$a_{\ell-1}$ (inclusive) which is not $(p-1)/2$ is in fact greater
than $(p-1)/2$.  The other letters $u_j$ of $\pat(a)$ are determined
similarly by looking at the $p$-adic digits of $a$ to the left of
$a_{\ell-1-j}$.  (Finding the first digit $\neq(p-1)/2$ may require
wrapping around.)

For example, if $\ell=4$ and $p>3$, then $\pat(a)=ffvf$ when
\[a=(p-1)/2 + (p-2)p+0p^2+(p-2)p^3.\]

The following proposition records the ``multiplicities'' of each
pattern.

\begin{prop}\label{prop:multiplicities} 
  Let $p$ be odd and $d=p^\ell-1$ and define $S = S_f \cup S_v$ as
  usual.
  \begin{enumerate}
  \item For $w\in\WW_\ell$, write $\mu(w)$ for the number of
    elements $a\in S$ with $\pat(a)=w$. Then
\[\mu(f^\ell)=\mu(v^\ell)=
  \left(\frac{p+1}2\right)^\ell-2.\] 
If $w$ has $k>0$ breaks, then
\[\mu(w)=
  \left(\frac{p-1}2\right)^{k}\left(\frac{p+1}2\right)^{\ell-k}.\]
  \item More generally,
    \[\mu(\text{---}f^e)=\mu(\text{---}v^e)
      =\left(\frac{p+1}2\right)p^{\ell-e}-2,\]
    \[\mu(\text{---}v^{e_k}\cdots f^{e_1})
         =\mu(\text{---}f^{e_k}\cdots v^{e_1})
      =\left(\frac{p+1}2\right)^{\sum e_j-k}\left(\frac{p-1}2\right)^{k-1}
      \left(\frac{p^{\ell+1-\sum e_j}-1}2\right),\]
    and 
    \[\mu(\text{---}f^{e_{k+1}}v^{e_k}\cdots f^{e_1})
         =\mu(\text{---}v^{e_{k+1}}f^{e_k}\cdots v^{e_1})
      =\left(\frac{p+1}2\right)p^{\sum e_j-k-1}\left(\frac{p-1}2\right)^{k}
      \left(\frac{p^{\ell+1-\sum e_j}+1}2\right).\]
  \end{enumerate}
\end{prop}

\begin{proof}
We prove part (1) and leave part (2) for the reader.
The number of elements $a\in S$ whose
  pattern is $f^\ell$ or $v^\ell$ is $(\frac{p+1}2)^\ell-2$.  Indeed,
  for $f^\ell$ (resp.\ $v^\ell$), we may choose each $a_j$ freely with
  $(p-1)/2\le a_j\le p-1$ (resp.\ $0\le a_j\le (p-1)/2$) except that
  we may not take them all to be $(p-1)/2$ nor all $p-1$ (resp.\ $0$).
  This proves the first claim in (1).
  
  Suppose $w$ is a pattern with breaks.  Then
  the inequalities at the beginning of this subsection show that 
  an element $a$ with pattern $w$ should have digits $a_j$ satisfying:
  \[
    \begin{cases}
      a_j\le(p-1)/2&\text{if $u_{\ell-1-j}=v$  and $j$ is not a break of $w$,}\\
      a_j<(p-1)/2&\text{if $u_{\ell-1-j}=v$  and $j$ is a break of $w$,}\\
      a_j\ge(p-1)/2&\text{if $u_{\ell-1-j}=f$  and $j$ is not a break of $w$,}\\
      a_j>(p-1)/2&\text{if $u_{\ell-1-j}=f$  and $j$ is a break of $w$.}
    \end{cases}\]
The count displayed in the second claim in (1) is then immediate.
\end{proof}


\begin{thm}\label{thm:EO-encompassing}
  Let $p$ be odd and $d=p^\ell-1$.  Let $\CC_d$ be the curve
  $y^d=x(1-x)$ and let $J_d$ be its Jacobian.  Then the Ekedahl--Oort
  type of $J_d$ has the form
    \[ [\nearrow^{\mu_0}\to^{\mu_1}\cdots\to^{\mu_{2^{\ell-1}-1}}]\]
    where $\mu_0=(\frac{p+1}2)^\ell-2$ and for
    $1\le i\le 2^{\ell-1}-1$, letting $k(i)$ be the function in
    Lemma~\ref{lemma:k},
    \[\mu_i=
      \begin{cases}
        \left(\frac{p-1}2\right)^{k(i)}\left(\frac{p+1}2\right)^{\ell-k(i)}
        &\text{if $i$ is even,}\\
        \left(\frac{p-1}2\right)^{k(i)+1}\left(\frac{p+1}2\right)^{\ell-k(i)-1}
        &\text{if $i$ is odd.}
      \end{cases}\]
\end{thm}

\begin{proof}
This follows immediately from Theorem~\ref{thm:FEO-general}, the
calculation of multiplicities in
Proposition~\ref{prop:multiplicities}, and the evaluation of the
number of breaks in Lemma~\ref{lemma:k}.  (One should note that the
$k$ appearing in Proposition~\ref{prop:multiplicities} for $w_i$ is
$k(i)$ if $i$ is even, and it is $k(i)+1$ if $i$ is odd.)
\end{proof}

\subsection{Examples} Suppose $p$ is odd.
\begin{enumerate}
\item If $\ell=1$, $\CC_d$ has genus $(p-3)/2$ and is ordinary.  The
  list of words as at \eqref{Ewords} is just $f$.  The elementary
  sequence 
is $[\nearrow^{(p-3)/2}]=[1,2,\dots,(p-3)/2]$.
\item If $\ell=2$, the list of words is $ff,fv$.  The elementary
  sequence has an increasing section of length $(p+1)^2/4-2$ and a
  constant section of length $(p-1)^2/4$, i.e., it is:
  \[  [\nearrow^{(p+1)^2/4-2}\to^{(p-1)^2/4}]
    =[1,2,\dots,(p+1)^2/4-2,\dots,(p+1)^2/4-2].\]
\item If $\ell=3$, the list of words is $f^3,f^2v,fvf,fvv$.  Letting
  $m=(p+1)^3/8-2$ and $n=(p+1)(p-1)^2/8$, the elementary sequence is
  $[\nearrow^{m}\to^{n} \nearrow^{n}\to^{n}]$.
  \item If $\ell = 4$, the word $f^3v$ occurs; this is the smallest
    example whose $BT_1$ group scheme was not previously known to
    occur as a factor of the $p$-torsion of a Jacobian, for all $p$.
\end{enumerate}

  
\begin{prop}
  Let $p$ be odd and let $d=p^\ell-1 > 2$.  
  Then:
    \begin{enumerate}
  \item the $p$-rank of $J_d[p]$ is $(\frac{p+1}2)^\ell-2$;
  \item the $a$-number of $J_d[p]$ is $\frac{p-1}2\frac{p^{\ell-1}-1}2$;
  \item the $s_{1,1}$-multiplicity of $J_d[p]$ is $0$ if $\ell$ is
    odd and is
    $(\frac{p^2-1}4)^{\ell/2}$ if $\ell$ is even;
  \item and the $u_{1,1}$-number of $J_d$ is the sum of the
    $s_{1,1}$-multiplicity and
    \[
        \sum_{j=0}^{\lfloor(\ell-4)/2\rfloor}
        \left(\frac{p+1}2\right)^2\left(\frac{p-1}2\right)^{2j+1}
        \left(\frac{p^{\ell-3-2j}-1}2\right).\]
  \end{enumerate}
\end{prop}

\begin{proof}
  This follows from Proposition~\ref{prop:FEO-general-numbers} using
  the multiplicities in Proposition~\ref{prop:multiplicities}.
\end{proof}

\section{$J_d[p]$ in the Hermitian case}\label{s:hermitian}

Suppose $p$ is odd.  Fix an integer $\lambda \geq 1$ and let
$d=p^\lambda+1$.  In this case, the Fermat curve of degree $d$ is
isomorphic (over $\Fpbar$) to the Hermitian curve $H_q$ with equation
$y_1^{q+1} = x_1^q+x_1$ where $q=p^\lambda$.  It is well known that
$H_q$ is supersingular and its Ekedahl--Oort type was studied in
\cite{PriesWeir15}.  Since $C_d$ is a quotient of $H_q$, it is also
supersingular in this case.

\subsection{$p$-adic digits}
Let $S= \Z/d\Z\setminus\{0,d/2\}$.  
Let $\pi:S\to S$ be induced by multiplication by $p$.  
Let
$S_v=\{b\in S \mid 0<b<d/2\}$ and $S_f=\{b\in S \mid d/2<b<d\}$. Let
\[S'=\left\{(b_1,\dots,b_{\lambda})\left| \ 0\le b_j\le p-1\text{ and not
        all $b_j=(p-1)/2$}\right.\right\}.\]
  There is a bijection $S'\to S$ given by
\[(b_1,\dots,b_\lambda)\mapsto b=1+\sum_{j=1}^\lambda b_jp^{j-1}.\]
  Under this bijection, the permutation $\pi$ is given by
  $(b_1,\dots,b_\lambda)\mapsto(p-1-b_\lambda,b_1,\dots,b_{\lambda-1})$.
  
  An element $b$ belongs to $S_v$ if and only if $b<(p^\lambda+1)/2$.
  This is true if and only if, in the tuple, the entry $b_j$ with
  largest $j$ such that $b_j \neq (p-1)/2$ has the property that
  $b_j < (p-1)/2$.

\subsection{Multiplicities}
The multiplicative order of $p$ modulo $d$ is $\ell=2\lambda$.  We
define a map $\pat':S\to\WW_\lambda$ as follows: Given $b\in S$, the
\emph{pattern} $\pat'(b)$ of $b$ is the word
$w=u_{\lambda-1}\cdots u_0$ given by
\[u_j=f \text{ if } p^jb\in S_f \text{ and } u_j=v \text{ if } p^jb\in S_v.\]
(The notation $\pat'$ is used to distinguish this from the pattern in
the encompassing case.)  If the word for $b$ has length $\ell$ (the
maximum length), then it is $\pat'(b)^c\cdot\pat'(b)$ (where the $c$
stands for the complementary word).  For any $b\in S$, the word for
$b$ has a power with length $\ell$ and this power equals
$\pat'(b)^c\cdot\pat'(b)$.  Note that since $p^\lambda=-1\pmod d$, the
word of $b$ is ``self-dual'', i.e., of the form $t^c t$, for every
$b \in S$.

For a word $w$ of length $\lambda$, let $\mu'(w)$ be the number of
elements $b\in S$ with $\pat'(b)=w$.  For a word $t$ of length
$\le\lambda$, let $\mu'(\text{---}t)$ be the number of elements
$b\in S$ with $\pat'(b)=t' \cdot t$ for some $t'$, in other words, the
number of $b$ with pattern ending in $t$.

\begin{prop}\label{prop:Hermitian-multiplicities}
\mbox{} Let $p$ be odd.
\begin{enumerate}
\item Suppose $e_1,\dots,e_k$ are positive
      integers with $\sum e_i=\lambda$.
If $k$ is odd, then
    \[\mu'(f^{e_k}v^{e_{k-1}}\cdots f^{e_1})
      =\mu'(v^{e_k}f^{e_{k-1}}\cdots v^{e_1})
      =\left(\frac{p+1}2\right)^{\lambda-k}\left(\frac{p-1}2\right)^k,\]
and if $k$ is even, then
    \[\mu'(v^{e_{k}}\cdots f^{e_1})
      =\mu'(f^{e_k}\cdots v^{e_1})
      =\left(\frac{p+1}2\right)^{\lambda+1-k}\left(\frac{p-1}2\right)^{k-1}.\]
  \item More generally, given integers $e_1, \ldots, e_k > 0$, let
    $\lambda'=\sum e_i$ and suppose $\lambda' \le \lambda$.  If $k$ is
    odd, and $t$ is a word of the form
    $t=f^{e_k}v^{e_{k-1}}\cdots v^{e_2}f^{e_1}$, then
    \[\mu'(\text{---}t)
      =\mu'(\text{---}t^c) =\left(\frac{p+1}2\right)^{\lambda'-k}
      \left(\frac{p-1}2\right)^{k-1}
      \left(\frac{p^{\lambda+1-\lambda'}-1}2\right),\]
     and if $k$ is even, and $t$ is a word of the form
    $t=v^{e_k} \cdots f^{e_1}$, then
    \[\mu'(\text{---}t)
      =\mu'(\text{---} t^c)
      =\left(\frac{p+1}2\right)^{\lambda'-k}
         \left(\frac{p-1}2\right)^{k-1}
          \left(\frac{p^{\lambda+1-\lambda'}+1}2\right).\]
  \end{enumerate}
\end{prop}

Part (1) of Proposition~\ref{prop:Hermitian-multiplicities} contradicts
\cite[Lemma~4.3]{PriesWeir15}, which we believe is in error.

\begin{proof}
Part (1) follows from part (2), so we
  will prove the latter.  It is clear that
  $\mu'(\text{---}t)=\mu'(\text{---}t^c)$.
  
  Suppose $k$ is odd and $t=f^{e_k}v^{e_{k-1}}\cdots v^{e_2}f^{e_1}$.
  Write $f^{e_k}\cdots f^{e_1}=u_{\lambda'-1}\cdots u_0$ with
  $u_j\in\{f,v\}$.  Then $b\in S$ has pattern $\text{---}t$ if and
  only if the $p$-adic digits $(b_1,\dots,b_{\lambda})$ satisfy, for
  $\lambda+1-\lambda' < j \leq \lambda$,
  \begin{equation} \label{Econditionb}
    \begin{cases}
      b_j\le(p-1)/2&\text{if $u_{\lambda-j}=v$  and $j$ is not a break of $t$,}\\
      b_j<(p-1)/2&\text{if $u_{\lambda-j}=v$ 
      and $j$ is a break of $t$,}\\
      b_j\ge(p-1)/2&\text{if $u_{\lambda-j}=f$  and $j$ is not a break of $t$,}\\
      b_j>(p-1)/2&\text{if $u_{\lambda-j}=f$  and $j$ is a break of $t$,}
    \end{cases}
    \end{equation}
and the number corresponding to the tuple 
$\beta=(b_1, \dots, b_{\lambda+1-\lambda'})$ is large, namely  
\[p^{\lambda+1-\lambda'}+1>
  1+\sum_{j=1}^{\lambda+1-\lambda'}b_jp^{j-1}
  >(p^{\lambda+1-\lambda'}+1)/2.\]
So there are $(p^{\lambda+1-\lambda'}-1)/2$ choices for $\beta$.
Taking the product with the number of possibilities for $b_j$ for
$\lambda+1-\lambda' < j \leq \lambda$ yields the quantity in the
statement.

Similarly, if $k$ is even and
$t=v^{e_k}\cdots f^{e_1}=u_{\lambda'-1}\cdots u_0$ with
$u_j\in\{f,v\}$, then $b\in S$ has pattern $\text{---}t$ if and only
if the $p$-adic digits $(b_1,\dots,b_{\lambda})$ satisfy
\eqref{Econditionb}, for $\lambda+1-\lambda'<j\le\lambda$
and the number corresponding to the tuple $\beta$ is small, namely
\[0<1+\sum_{j=1}^{\lambda+1-\lambda'}b_jp^{j-1}
  \le (p^{\lambda+1-\lambda'}+1)/2.\]
So there are $(p^{\lambda+1-\lambda'}+1)/2$ choices for $\beta$.
Again, taking the product with the number of possibilities for $b_j$
for $\lambda+1-\lambda' < j \leq \lambda$ yields the quantity in the
statement.
\end{proof}


\begin{thm}\label{thm:EO-Hermitian}
  Let $p$ be odd and $d=p^\lambda+1$. Let $\CC_d$ be the curve
  $y^d=x(1-x)$ and let $J_d$ be its Jacobian.  Then the Ekedahl--Oort
  type of $J_d$ has the form
    \[ [\to^{\mu'_0}\nearrow^{\mu'_1}\cdots\to^{\mu'_{2^{\lambda-1}-1}}]\]
    where, letting $k(i)$ be the function described in Lemma~\ref{lemma:k},
    \[\mu'_i=
      \begin{cases}
      \left(\frac{p+1}2\right)^{\lambda-k(i)-1}\left(\frac{p-1}2\right)^{k(i)+1}
      &\text{if $i$ is even,}\\
      \left(\frac{p+1}2\right)^{\lambda-k(i)}\left(\frac{p-1}2\right)^{k(i)}
      &\text{if $i$ is odd.}\\      
      \end{cases}\]
\end{thm}

\begin{proof}
  This follows immediately from Theorem~\ref{thm:FEO-general}, the
  multiplicities in Proposition~\ref{prop:Hermitian-multiplicities},
  and the number of breaks in Lemma~\ref{lemma:k}.  (Note that the $k$
  in Proposition~\ref{prop:Hermitian-multiplicities} for $w_i$ is
  $k(i)+1$.)
\end{proof}

\subsection{Examples}
Let $p$ be odd.
\begin{enumerate}
\item If $\lambda=1$, the curve $\CC_d$ has genus $(p-1)/2$ and is
  superspecial: the list of words starting with $f$ and with positive
  multiplicity is $fv$, and the elementary sequence is
  $[\to^{(p-1)/2}]=[0,\dots,0]$.
\item If $\lambda=2$, the list of words is $ffvv$, $fvvf$, and
  the elementary sequence has a constant section of length $(p^2-1)/4$
  and an increasing section of length $(p^2-1)/4$:
  \[  [\to^{(p^2-1)/4}\nearrow^{(p^2-1)/4}]
    =[0,\dots,0,1,2,\dots,(p^2-1)/4].\]
\item If $\lambda=3$, the list of words is
  $f^3v^3,f^2v^3f,(fv)^3,fv^3f^2$, and the elementary sequence has four
  segments and has the form $[\to^m\nearrow^m\to^n\nearrow^m]$ where
  $m=(p+1)^2(p-1)/8$ and $n=(p-1)^3/8$.
  \end{enumerate}
  
\begin{prop}
  Let $p$ be odd, let $\lambda$ be a positive integer, and let
  $d=p^\lambda +1$.  Then:
\begin{enumerate}
\item the $p$-rank of $J_d[p]$ is $0$;
\item the $a$-number of $J_d[p]$ is $(p-1)(p^{\lambda-1}+1)/4$;
\item the $s_{1,1}$-multiplicity of $J_d[p]$ is $0$ if $\lambda$ is
  even and $(\frac{p-1}2)^{\lambda}$ if $\lambda$ is odd;
\item and the $u_{1,1}$-number of $J_d[p]$ is the sum of the
  $s_{1,1}$-multiplicity and
\[\sum_{j=0}^{\lfloor(\lambda-4)/2\rfloor}
    \left(\frac{p+1}2\right)^2\left(\frac{p-1}2\right)^{2j+1}
    \left(\frac{p^{\lambda-3-2j}+1}2\right)\\
    +\begin{cases}
      0&\text{if $\lambda=1$,}\\
      \left(\frac{p+1}2\right)^2\left(\frac{p-1}2\right)^{\lambda-2}
                    &\text{if $\lambda>1$ and odd,}\\
                   \left(\frac{p+1}2\right)\left(\frac{p-1}2\right)^{\lambda-1}
                    &\text{if $\lambda$ even.}\\
    \end{cases}\]

\end{enumerate}
\end{prop}

The analogue of the $a$-number calculation in part (2) for the Fermat
curve of degree $d=p^\lambda+1$ is given in
\cite[Prop.~14.10]{Gross90}.

\begin{proof}
  This follows from Proposition~\ref{prop:FEO-general-numbers} using
  the multiplicities in
  Proposition~\ref{prop:Hermitian-multiplicities}.



\end{proof}

\bibliography{database}

\end{document}